\documentclass[a4paper,11pt]{amsart}
\usepackage{amssymb, amsmath,mathrsfs}
\usepackage[margin=10pt,font=footnotesize,labelfont=bf,labelsep=endash]{caption}

\usepackage{graphicx}
\usepackage{amsmath}
\usepackage{amssymb}

\newtheorem{thm}{Theorem}
\newtheorem{cor}{Corollary}%
\newtheorem{lem}{Lemma}

\theoremstyle{remark}
\newtheorem{remark}{Remark}

\theoremstyle{definition}

\theoremstyle{plain}

\numberwithin{equation}{section}

\def\CC{{\mathbb C}}

\def\HH{{\mathbb H}}
\def\NN{{\mathbb N}}

\def\spn{\operatorname{span}}

\def\veca{{\text{\boldmath$a$}}}

\def\vecb{{\text{\boldmath$b$}}}

\def\vece{{\text{\boldmath$e$}}}

\def\vecm{{\text{\boldmath$m$}}}

\def\vecr{{\text{\boldmath$r$}}}

\def\vecu{{\text{\boldmath$u$}}}
\def\vecv{{\text{\boldmath$v$}}}

\def\vecw{{\text{\boldmath$w$}}}

\def\vecx{{\text{\boldmath$x$}}}

\def\vecy{{\text{\boldmath$y$}}}

\def\veczeta{{\text{\boldmath$\zeta$}}}

\def\scrB{{\mathcal B}}

\def\scrD{{\mathcal D}}

\def\scrL{{\mathcal L}}

\def\scrR{{\mathcal R}}

\def\dim{\operatorname{dim}}

\def\S{\operatorname{S{}}}

\def\SL{\operatorname{SL}}

\def\SO{\operatorname{SO}}

\def\vol{\operatorname{vol}}

\def\trans{\,^\mathrm{t}\!}

\def\Onder#1#2#3#4#5{#1 \setbox0=\hbox{$#1$}\setbox1=\hbox{$#2$}
       \dimen0=.5\wd0 \dimen1=\dimen0 \dimen2=\dp0 \dimen3=\dimen2
       \advance\dimen0 by .5\wd1 \advance\dimen0 by -#4
       \advance\dimen1 by -.5\wd1 \advance\dimen1 by -#4
       \advance\dimen2 by -#3 \advance\dimen2 by \ht1
       \advance\dimen2 by 0.3ex \advance\dimen3 by #5
        \kern-\dimen0\raisebox{-\dimen2}[0ex][\dimen3]{\box1}
       \kern\dimen1}
\def\utilde#1{\Onder{#1}{\mbox{\char'176}}{0pt}{0ex}{0.5ex}} %
\def\myutilde#1{\Onder{#1}{\mbox{$\sim$}}{0pt}{0ex}{0.5ex}} %

\newcommand{\R}{\mathbb{R}}
\newcommand{\Z}{\mathbb{Z}}
\newcommand{\N}{\mathbb{N}}

\newcommand{\HSm}{{\S_\pm^{n-1}}}
\renewcommand{\aa}{\mathsf{a}}
\newcommand{\kk}{\mathsf{k}}
\newcommand{\nn}{\mathsf{n}}
\newcommand{\sfrac}[2]{{\textstyle \frac {#1}{#2}}}
\newcommand{\col}{\: : \:}
\newcommand{\Si}{\mathcal{S}}
\newcommand{\bn}{\mathbf{0}}
\newcommand{\ta}{\utilde{a}}
\newcommand{\tu}{\utilde{u}}
\newcommand{\tM}{\myutilde{M}}
\newcommand{\tkk}{\utilde{\mathsf{k}}}

\newcommand{\ve}{\varepsilon}
\newcommand{\F}{\mathcal{F}}
\newcommand{\FC}{\mathcal{C}}
\newcommand{\FG}{\mathcal{G}}
\newcommand{\matr}[4]{\left( \begin{matrix} #1 & #2 \\ #3 & #4 \end{matrix} \right) }

\title{On the limit distribution of Frobenius numbers}
\author{Andreas Str\"ombergsson}
\address{Department of Mathematics, Box 480, Uppsala University,
SE-75106 Uppsala, Sweden\newline
\rule[0ex]{0ex}{0ex} \hspace{8pt}{\tt astrombe@math.uu.se}}
\date{\today}

\begin{document}

\begin{abstract}
The Frobenius number $g(\veca)$ of an integer vector $\veca$ with positive
coprime coefficients is defined as the largest integer that does not have a
representation as a non-negative integer linear combination of the 
coefficients of $\veca$.
According to a recent result by Marklof, if $\veca$ is taken to be random
in an expanding $d$-dimensional domain $\scrD$, 
then $(a_1\cdots a_d)^{-1/(d-1)}g(\veca)$ %
has a limit distribution.
In the present paper we prove an asymptotic formula for the
(algebraic) tail behavior of this limit distribution.
We also prove that the corresponding upper bound on the probability of
the Frobenius number %
being large holds uniformly with respect to the 
expansion factor of the domain $\scrD$.
Finally we prove that for large $d$, the limit distribution of 
$(a_1\cdots a_d)^{-1/(d-1)}g(\veca)$ %
has almost all of its mass concentrated between
$(d-1)!^{1/(d-1)}$ and $1.757\cdot(d-1)!^{1/(d-1)}$.
The techniques involved in the proofs come from the geometry of numbers,
and in particular we use results by Schmidt on the distribution of
sublattices of $\Z^m$, and bounds by Rogers and Schmidt
on lattice coverings of space with convex bodies.
\end{abstract}

\maketitle

\section{Introduction}

We denote by $\widehat\N^d$ the set of integer vectors in $\R^d$ with positive
coprime coefficients (viz.\ the greatest common divisor of all
coefficients is one).
Given $\veca=(a_1,\ldots,a_d)\in\widehat\N^d$, %
the Frobenius number $g(\veca)=g(a_1,\ldots,a_d)$ is defined as the largest 
integer which is not representable as a non-negative integer 
combination of $a_1,\ldots,a_d$.
The problem of computing $g(\veca)$ is known as the Frobenius problem
or the coin exchange problem, 
and it has been studied extensively.
Cf., e.g., \cite{aR2005} and \cite[Problem C7]{rG2004a}.

In the majority of problems related to Frobenius numbers,
it is more convenient to consider the function
\begin{align}
f(\veca)=f(a_1,\ldots,a_d)=g(a_1,\ldots,a_d)+a_1+\ldots+a_d.
\end{align}
Clearly, $f(\veca)$ is the largest integer which is not a 
\textit{positive} integer combination of $a_1,\ldots,a_d$.

In the case of two variables, $d=2$, the Frobenius number is given
by Sylvester's formula (\cite[Theorem 2.1.1]{aR2005}),
\begin{align}
g(a_1,a_2)=a_1a_2-a_1-a_2
\qquad(\text{viz., }\: f(a_1,a_2)=a_1a_2).
\end{align}
For $d\geq3$ no explicit formula is known.
Arnold 
(\cite{vA99},
\cite{vA2006},
\cite{vA2007})
asked about the behavior of $g(a_1,\ldots,a_d)$ for
a 'random' large vector $(a_1,\ldots,a_d)\in\R^d$.
Davison had previously asked similar questions for $d=3$, in
\cite[Sec.\ 5]{jD94}.
Recently Marklof
(\cite{jM2009}) 
obtained a definitive result for arbitrary $d\geq3$,
generalizing previous results by
Bourgain and Sinai \cite{zByS2007} in the case $d=3$
(cf.\ also 
Shchur, Sinai, Ustinov \cite{vSySaU2009}):
\begin{thm}\label{JENSTHM} 
(Marklof \cite{jM2009}).
Given $d\geq3$, there exists a continuous non-increasing function
$\Psi_d:\R_{\geq0}\to\R_{\geq0}$ with $\Psi_d(0)=1$, such that for
any bounded set $\scrD\subset\R_{\geq0}^d$ with nonempty interior and
boundary of Lebesgue measure zero, and any $R\geq0$,
\begin{align}\label{JENSTHMRES}
\lim_{T\to\infty}\:\frac1{\#(\widehat\N^d\cap T\scrD)}
\#\Bigl\{\veca\in\widehat\N^d\cap T\scrD
\col\frac{f(\veca)}{(a_1\cdots a_d)^{1/(d-1)}}>R\Bigr\}
=\Psi_d(R).
\end{align}
\end{thm}
For arbitrary $d\geq3$, Li 
\cite[Thm.\ 1.3]{hL2010}
has recently obtained an effective version
of Theorem~\ref{JENSTHM}, where \eqref{JENSTHMRES}
is proved to hold with a power convergence rate (w.r.t.\ $T$).%

\begin{figure}
\begin{center}

\framebox{
\begin{minipage}{0.45\textwidth}
\unitlength0.1\textwidth
\begin{picture}(10,7.0)(0,1.2)
\put(-0.3,1){\includegraphics[width=1.05\textwidth,height=0.73\textwidth]{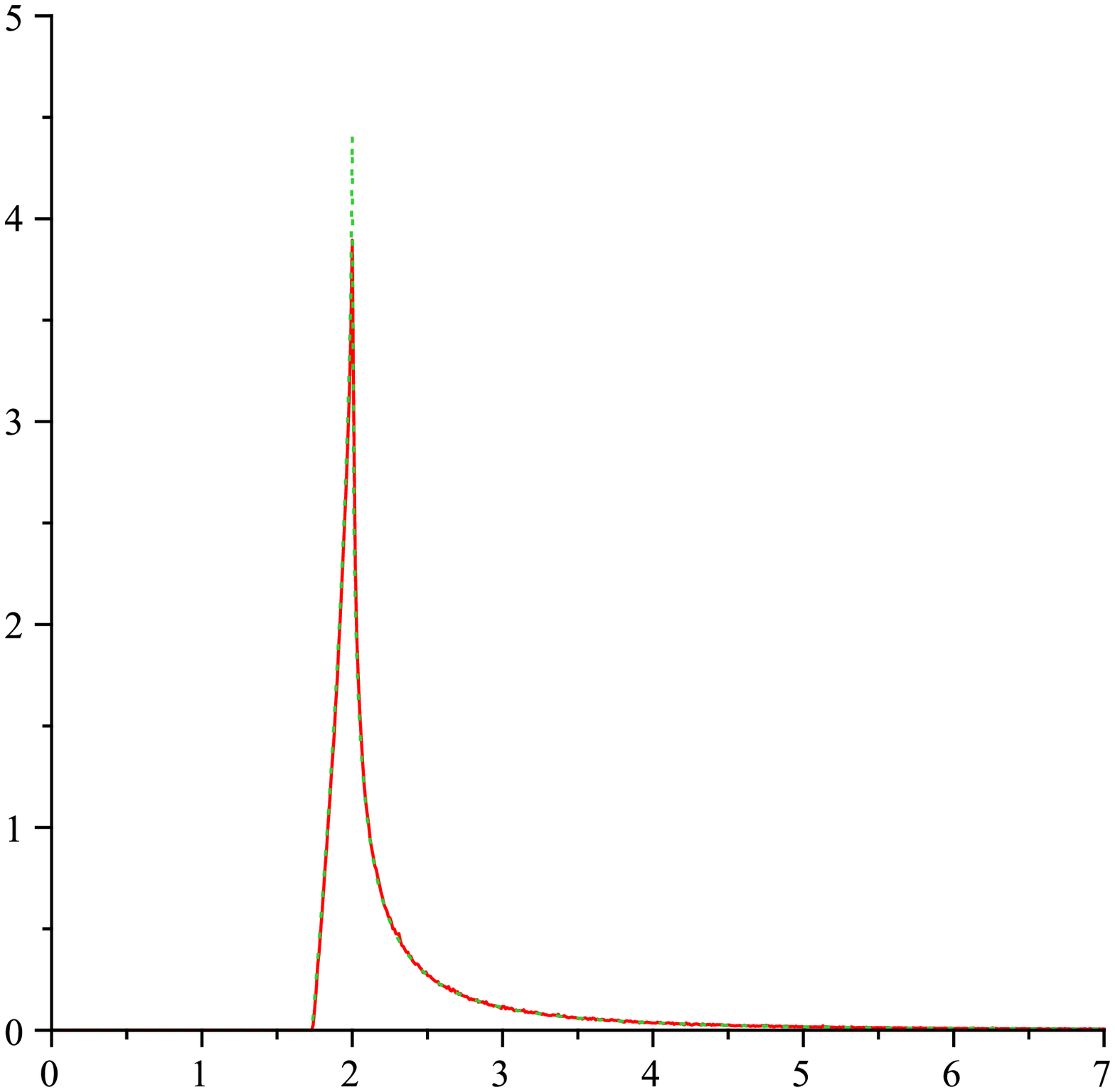}}
\put(8.4,7.3){$d=3$} %
\put(9.7,2.0){\begin{scriptsize}$R$\end{scriptsize}}
\end{picture}
\end{minipage}
}
\framebox{
\begin{minipage}{0.45\textwidth}
\unitlength0.1\textwidth
\begin{picture}(10,7.0)(0,1.2)
\put(-0.3,1){\includegraphics[width=1.05\textwidth,height=0.73\textwidth]{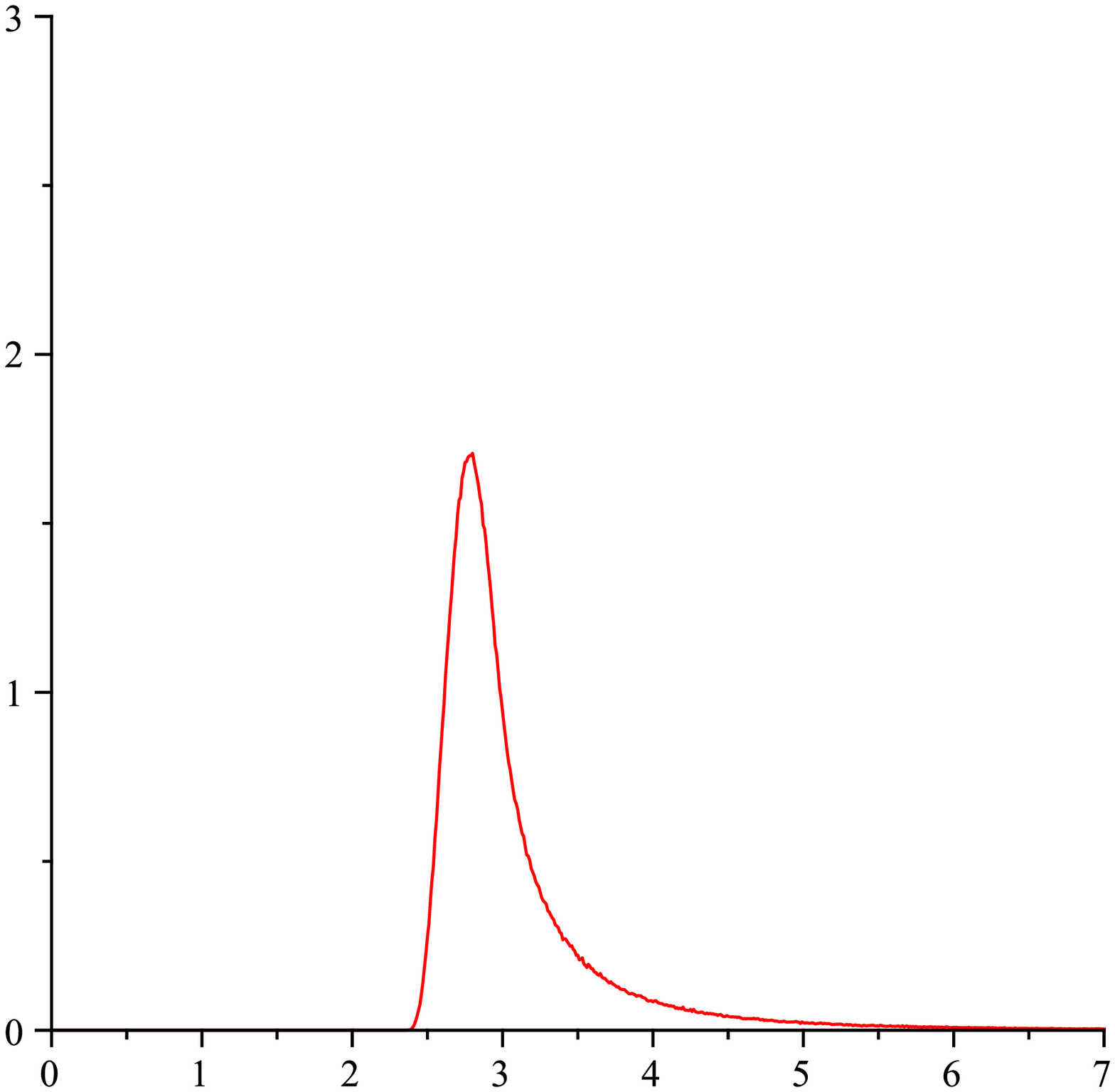}}
\put(8.4,7.3){$d=4$}  %
\put(9.7,2.0){\begin{scriptsize}$R$\end{scriptsize}}
\end{picture}
\end{minipage}
}
\vspace{3pt}

\framebox{
\begin{minipage}{0.45\textwidth}
\unitlength0.1\textwidth
\begin{picture}(10,7.0)(0,1.2)
\put(-0.3,1){\includegraphics[width=1.05\textwidth,height=0.73\textwidth]{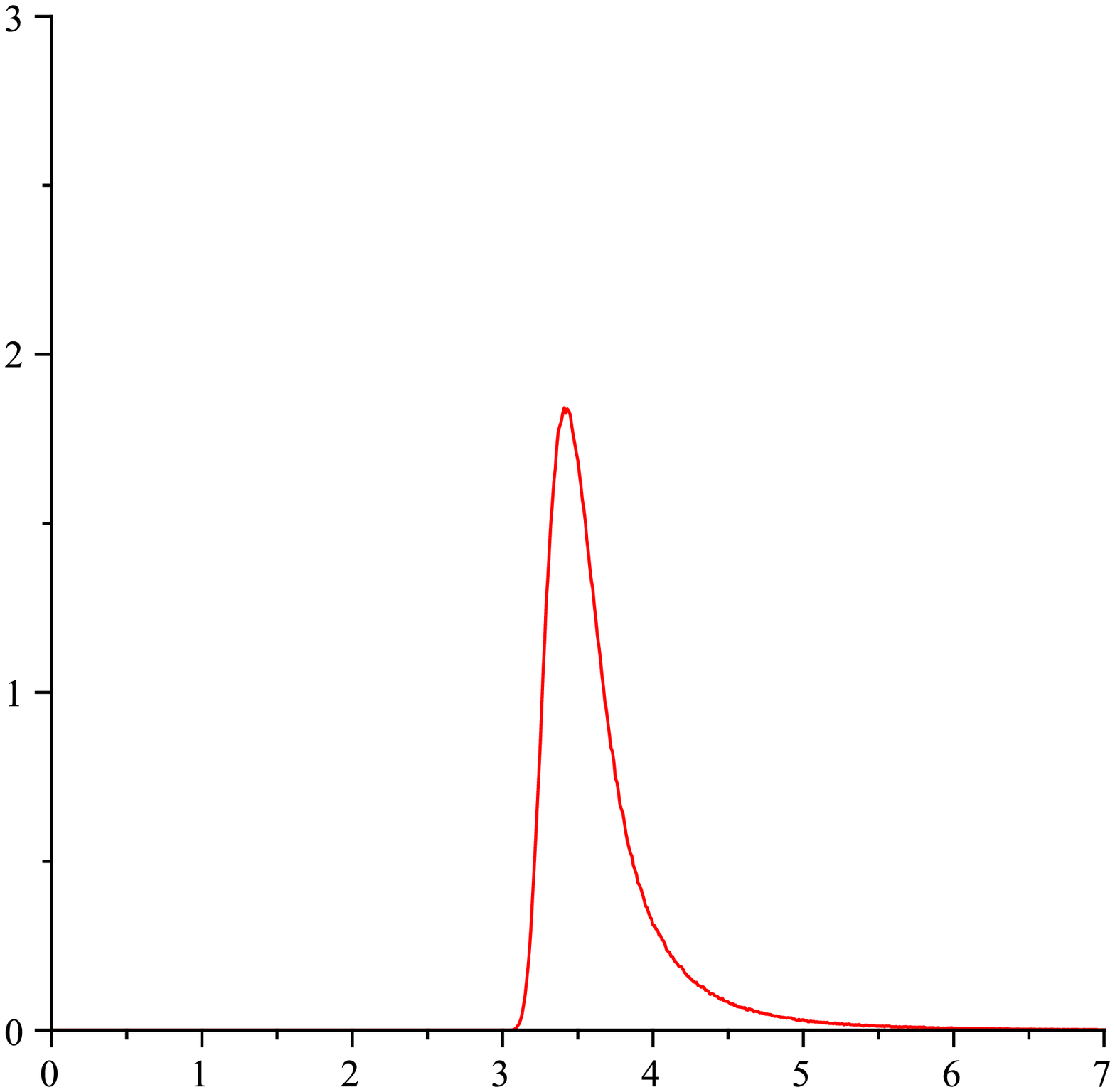}}
\put(8.4,7.3){$d=5$} %
\put(9.7,2.0){\begin{scriptsize}$R$\end{scriptsize}}
\end{picture}
\end{minipage}
}
\framebox{
\begin{minipage}{0.45\textwidth}
\unitlength0.1\textwidth
\begin{picture}(10,7.0)(0.0,1.2)
\put(-0.3,1){\includegraphics[width=1.05\textwidth,height=0.73\textwidth]{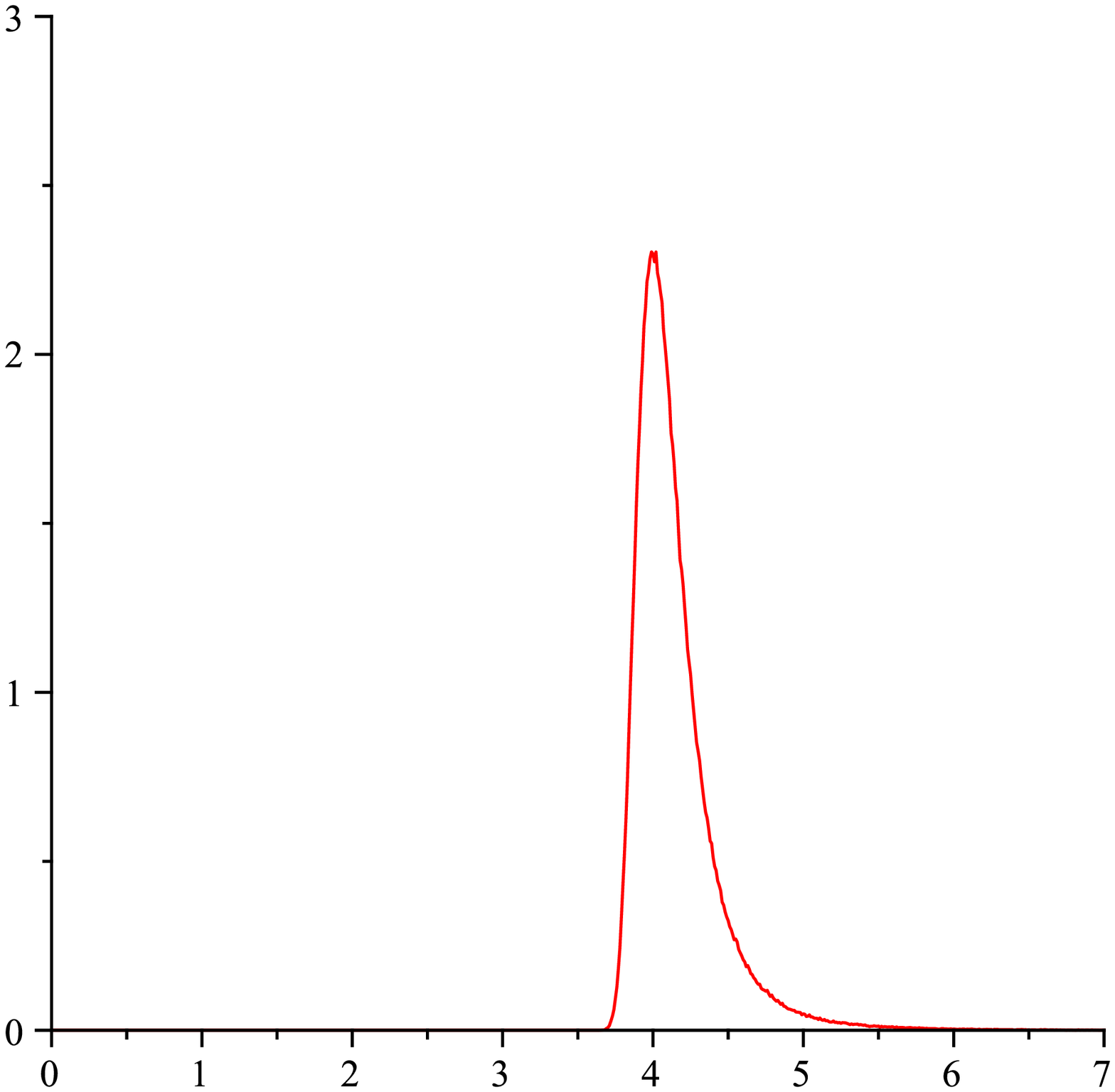}}
\put(8.4,7.3){$d=6$} %
\put(9.7,2.0){\begin{scriptsize}$R$\end{scriptsize}}
\end{picture}
\end{minipage}
}
\end{center}
\caption{Experimental graphs of the density functions \label{FIGURE}
$\psi_d(R)=-\frac d{dR}\Psi_d(R)$ 
of the limit distribution in Theorem \ref{JENSTHM}, for $d=3,4,5,6$. 
The graphs were obtained by computing
$(a_1\cdots a_d)^{-\frac1{d-1}}f(\veca)$ for
$1.2\cdot 10^6$ integer vectors $\veca$ picked at random in
$\widehat\NN^d\cap[0,T]^d$
with $T=10^{15}$, and collecting the results into bins of width $0.01$
along the $R$-axis.
The computations of $f(\veca)$ were performed using
the Frobby software package by Roune \cite{frobby}; cf.\ also \cite{roune2008}.
We repeated the computations using other random seeds and/or changing
$T$ to $10^{14}$, as well as to $10^{13},10^{12},10^{11}$ in some cases,
and the resulting graphs were consistently found to be practically
indistinguishable, except for %
$d=3$ and $R$ very near $2$.
For $d=3$ also the graph of the exact function in \eqref{PSI3EXPL}
is drawn (the dotted curve, which is distinguishable from the experimental
graph only for $R$ very near $2$).}
\end{figure}

Marklof also proved an explicit formula for 
$\Psi_d(R)$, namely that $\Psi_d(R)$ equals 
the probability that the simplex 
\begin{align}\label{DELTADEF}
\Delta=\bigl\{\vecx\in\R_{\geq0}^{d-1}\col\vecx\cdot\vece\leq1\bigr\},
\qquad\vece:=(1,1,\ldots,1),
\end{align}
has covering radius larger than $R$ with respect to a random 
lattice $L\subset\R^{d-1}$ of covolume one.
In other words (\cite[Thm.\ 2]{jM2009}),
\begin{align}\label{PSIFORMULA}
\Psi_d(R)=\mu_{d-1}\bigl(\bigl\{L\in X_{d-1}\col
\rho(L)>R\bigr\}\bigr),
\end{align}
where $X_{d-1}$ is the set of all lattices $L\subset\R^{d-1}$
of covolume one, 
$\mu_{d-1}$ is Siegel's measure (\cite{cS45a})
on $X_{d-1}$, normalized to be a probability measure,
and $\rho(L)$ is the covering radius of $\Delta$ with respect to $L$, viz.
\begin{align}
\rho(L)=\inf\bigl\{\rho>0\col L+\rho\Delta=\R^{d-1}\bigr\}.
\end{align}
In the special case $d=3$, Ustinov \cite{aU2010a}
(cf.\ also \cite{aU2009})
proved a more precise version of \eqref{JENSTHMRES}, %
where the averaging is performed over
only two of the three arguments
$a_1,a_2,a_3$, and the limit is obtained with a power rate of convergence.
Ustinov in fact gave a completely explicit formula for 
the limit density $\psi_3(R)=-\frac d{dR}\Psi_3(R)$
in terms of elementary functions:
\begin{align}\label{PSI3EXPL}
\psi_3(R)=
\begin{cases}
0 & (0\leq R \leq \sqrt 3)  \\
\frac{12}{\pi}\big(\frac{R}{\sqrt 3}-\sqrt{4-R^2}\big) & (\sqrt 3\leq R \leq 2)  \\
\frac{12}{\pi^2}\big( R\sqrt 3 \arccos\big(\frac{R+3\sqrt{R^2-4}}{4\sqrt{R^2-3}}\big)+\frac32 \sqrt{R^2-4}  \log\big(\frac{R^2-4}{R^2-3}\big)\big) & (R>2).
\end{cases}
\end{align}
See also \cite{multiloop}
for a derivation of \eqref{PSI3EXPL} from \eqref{PSIFORMULA}.

Our purpose in the present note is to discuss the behavior of 
$\Psi_d(R)$ for $d$ fixed and $R$ large, as well as for $d$ large.
For fixed $d\geq3$, 
it was proved by Li \cite{hL2010} that $\Psi_d(R)\ll_d R^{-(d-1)}$
for all $R>0$, and Marklof in an unpublished note
\cite{jM2010b}
pointed out that a corresponding lower bound
also holds: $\Psi_d(R)\gg_d R^{-(d-1)}$ for all $R\geq1$.
Our first result, which we will prove in section \ref{PSIDASYMPTSEC},
is an asymptotic formula refining these bounds:
\begin{thm}\label{PSIDASYMPTTHM}
Let $d\geq3$. Then
\begin{align}\label{PSIDASYMPTTHMRES}
\Psi_d(R)=\frac{d}{2\zeta(d-1)}R^{-(d-1)}+O_d(R^{-d-\frac1{d-2}})
\qquad\text{as }\: R\to\infty.
\end{align}
Here the error term is sharp; in fact there exists a constant $c>0$
which only depends on $d$, such that for all sufficiently large $R$,
\begin{align}\label{PSIDASYMPTTHMRES2}
\Psi_d(R)>\frac{d}{2\zeta(d-1)}R^{-(d-1)}+cR^{-d-\frac1{d-2}}.
\end{align}
\end{thm}
In particular we may note that %
\eqref{PSI3EXPL} implies
$\Psi_3(R)=\frac9{\pi^2}R^{-2}+\frac{33}{2\pi^2}R^{-4}+O(R^{-6})$ 
as $R\to\infty$,
which is consistent with Theorem \ref{PSIDASYMPTTHM}.

\vspace{5pt}

Combining Theorems \ref{JENSTHM} and \ref{PSIDASYMPTTHM}
we conclude that if $R$ is large, and if $\veca$ is picked at random
from %
a set of the type $\widehat\N^d\cap T\scrD$ with $T$ 
sufficiently large 
--- where the notion of ``sufficiently large'' may depend on $R$ ---
then the probability that the normalized Frobenius number
$\frac{f(\veca)}{(a_1\cdots a_d)^{1/(d-1)}}$ is greater than $R$ 
is approximately $\frac d{2\zeta(d-1)}R^{-(d-1)}$.
It is an interesting problem to try to get a more uniform
control on the probability of 
$\frac{f(\veca)}{(a_1\cdots a_d)^{1/(d-1)}}$ being large,
i.e.\ to give bounds from above and below,
\textit{uniformly with respect to large $T$ and $R$,} on %
\begin{align}
P_d(T,R):=\frac1{\#(\widehat\N^d\cap T\scrD)}
\#\Bigl\{\veca\in\widehat\N^d\cap T\scrD
\col\frac{f(\veca)}{(a_1\cdots a_d)^{1/(d-1)}}>R\Bigr\}.
\end{align}

Results related to this question have recently been obtained by 
Aliev and Henk \cite{iAmH2008} 
and
Aliev, Henk and Hinrichs \cite{iAmHaH2011},
by making use of Schmidt's results on the distribution
of similarity classes of sublattices of $\Z^m$, 
\cite{wS98}.
We will show %
that the application of 
\cite{wS98}
can be refined
--- using in particular the strong uniform %
error bounds which Schmidt provides for his asymptotic formulas ---
so as to give a uniform bound which 
significantly improves upon the bounds obtained
in \cite{iAmH2008}, \cite{iAmHaH2011},
and which can be viewed as a \textit{$T$-uniform} version of 
Li's upper bound $\Psi_d(R)\ll_d R^{-(d-1)}$.

For technical reasons we will consider the Frobenius number
normalized not with the factor $(a_1\cdots a_d)^{-1/(d-1)}$,
but with $s(\veca)^{-1}$, where
\begin{align}
s(\veca):=\frac{\sum_{j=1}^da_j\sqrt{\|\veca\|^2-a_j^2}}{\|\veca\|^{1-1/(d-1)}},
\end{align}
with $\|\veca\|$ denoting the standard Euclidean norm of $\veca$.
Thus, we set: %
\begin{align}\label{PDTRDEF}
\widetilde P_d(T,R):=\frac1{\#(\widehat\N^d\cap T\scrD)}
\#\Bigl\{\veca\in\widehat\N^d\cap T\scrD\col
\frac{f(\veca)}{s(\veca)}>R\Bigr\}.
\end{align}
Note that $P_d(T,R)$ and $\widetilde P_d(T,R)$
are defined for any $T>0$ such that 
$\widehat\N^d\cap T\scrD\neq\emptyset$;
in particular, for any fixed $\scrD\subset\R_{\geq0}^d$ with 
non-empty interior, $P_d(T,R)$ and $\widetilde P_d(T,R)$ are 
defined for all $T\gg_\scrD 1$.

The normalizing factor $s(\veca)$ was used also in 
Aliev and Henk, 
\cite{iAmH2008};
cf.\ also Fukshansky and Robins, \cite{lFsR2007}.
Note that if we assume that the coefficients of $\veca$ %
are ordered so that $a_1\leq a_2\leq\ldots\leq a_d$ then
$s(\veca)\asymp_d a_{d-1}a_d^{\frac1{d-1}}$;
in particular we have 
\begin{align}\label{OBVIOUSSABOUND}
(a_1\cdots a_d)^{\frac1{d-1}}\ll_d s(\veca)\ll_d
\|\veca\|^{\frac d{d-1}},\qquad\forall\veca\in\R_{>0}^d.
\end{align}
Hence there exists a constant $c_1>0$ which only
depends on $d$ such that
\begin{align}
\widetilde P_d(T,c_1R)\leq P_d(T,R),
\end{align}
for any $R>0$ and any $\scrD\subset\R_{\geq0}^d$ and $T>0$ such that 
$\widehat\N^d\cap T\scrD\neq\emptyset$.
On the other hand, if $\scrD$ is bounded and satisfies
$\overline\scrD\subset\R_{>0}^d$, then 
$s(\veca)\asymp (a_1\cdots a_d)^{1/(d-1)}$ holds uniformly over all
$\veca\in\R_{>0}\scrD$,
and thus we have
$P_d(T,R)\leq\widetilde P_d(T,c_2R)$
for all $T,R>0$ with $\widehat\N^d\cap T\scrD\neq\emptyset$, 
where $c_2>0$ is a constant which
only depends on $\scrD$.
Hence for any such region $\scrD$,
any of the two functions $P_d(T,R)$ and $\widetilde P_d(T,R)$ can
essentially be bounded in terms of the other,
as long as we allow an implied constant which may depend on $\scrD$.

Our main result on $\widetilde P_d(T,R)$ is the following bound,
which we will prove in Section \ref{UNIFSEC}.
\begin{thm}\label{UNIFTHM}
Let $d\geq3$, and let $\scrD\subset\R_{\geq0}^d$ be bounded
with nonempty interior. Then
\begin{align}\label{UNIFTHMRES1}
\widetilde P_d(T,R)\ll_{d,\scrD} R^{-(d-1)}, %
\end{align}
uniformly over all $T>0$ with $\widehat\N^d\cap T\scrD\neq\emptyset$,
and all $R>0$. Furthermore,
for any such $T$,
\begin{align}\label{UNIFTHMRES2}
\widetilde P_d(T,R)=0\qquad\text{whenever }\: R\geq
\bigl(T\sup_{\vecx\in\scrD}\|\vecx\|\bigr)^{1-\frac1{d-1}}.
\end{align}
\end{thm}
Theorem \ref{UNIFTHM} strengthens %
the bound
$\widetilde P_d(T,R)\ll R^{-2}$ which was given %
in \cite[Thm.\ 1.1]{iAmH2008}.
Note also that if the set $\scrD$ satisfies
$\overline\scrD\subset\R_{>0}^d$, then by the previous discussion
Theorem \ref{UNIFTHM} implies %
$P_d(T,R)\ll_{d,\scrD}R^{-(d-1)}$.

\vspace{5pt}

From many points of view, the normalization factor
$(a_1\cdots a_d)^{-1/(d-1)}$ is %
the most natural one to use in the Frobenius problem.
A clear indication of this is for example the fact that the
limit distribution obtained in Theorem \ref{JENSTHM} 
is independent of the choice of $\scrD$.
Hence it is interesting to ask whether the
bound in Theorem \ref{UNIFTHM} is valid also for $P_d(T,R)$,
\textit{without} %
the extra assumption $\overline\scrD\subset\R_{>0}^d$.
We conjecture that this is so. However in the present paper we will
content ourselves with pointing out a weaker bound, %
which follows fairly directly from Theorem \ref{UNIFTHM}
by an argument along the lines of 
\cite{iAmHaH2011},
and which strengthens %
the bound\footnote{We here correct for a mistake in
\cite[p.\ 530, lines 5-6]{iAmHaH2011}
by adding $\ve$ in the exponent: In the notation of \cite{iAmHaH2011}, 
the choice of ``$t=\frac{n-1}{n+1}$'' yields the bound
``$\beta^{-2\frac{(n-1)^2}{n(n+1)}}$'' and not ``$\beta^{-2\frac{n-1}{n+1}}$''
as claimed;
choosing $t$ optimally   %
yields the bound ``$\beta^{-2\frac{(n-1)^2}{n^2+1}}$'',
and using also 
\cite[p.\ 529, Remark 1]{iAmHaH2011} brings the bound down to
``$\beta^{-2\frac{n-1}{n+1}+\ve}$''.}
$P_d(T,R)\ll R^{-2\frac{d-1}{d+1}+\ve}$ obtained in 
\cite{iAmHaH2011}.

\begin{cor}\label{UNIFTHMCOR}
Let $d\geq3$, and let $\scrD\subset\R_{\geq0}^d$ be bounded
with nonempty interior. Then
\begin{align}\label{UNIFTHMCORRES1}
P_d(T,R)\ll_{d,\scrD} R^{-\frac12(d-1)}(\log(R+2))^{\frac12(d-3)}
\end{align}
uniformly over all $T>0$ with $\widehat\N^d\cap T\scrD\neq\emptyset$,
and all $R>0$. Furthermore,
\begin{align}\label{UNIFTHMCORRES2}
P_d(T,R)=0\qquad\text{whenever }\: R\geq
d\bigl(T\sup_{\vecx\in\scrD}\|\vecx\|\bigr)^{1-\frac1{d-1}}.
\end{align}
\end{cor}
We remark that in the special case $d=3$, it follows from Ustinov
\cite[pp.\ 1025, 1044]{aU2010a}
that the stronger bound $P_3(T,R)\ll_\scrD R^{-2}$ is valid at least so
long as we keep $T\gg R^{22+\ve}$.

It is also interesting to consider the \textit{moments}
of the (normalized) Frobenius number;
in particular the \textit{expected value} has been considered by many authors,
cf., e.g., 
\cite{iAmHaH2011},
\cite{vA99},
\cite{vA2006},
\cite{vA2007},
\cite[Sec.\ 5]{jD94},
\cite{aU2009}.
Note that it follows from Theorem \ref{PSIDASYMPTTHM}
(or just from the upper and lower bounds %
by Li \cite{hL2010} and Marklof \cite{jM2010b})
that the limit distribution described by $\Psi_d(R)$ possesses $k$th moment
for $k=1,\ldots,d-2$, and for no larger (integer) $k$.
Let us write $M_{d,k}$ for this moment:
\begin{align}\label{MDKDEF}
M_{d,k}:=-\int_0^\infty R^k\,d\Psi_d(R)
=k\int_0^\infty R^{k-1}\,\Psi_d(R)\,dR,\qquad
k=1,\ldots,d-2.
\end{align}
Now the following is an easy consequence of
Theorem \ref{JENSTHM} combined with Theorem \ref{UNIFTHM} and
Corollary \ref{UNIFTHMCOR}.
\begin{cor}\label{MOMENTCOR}
Let $d\geq3$, and let $\scrD\subset\R_{\geq0}^d$ be a bounded set
with nonempty interior and boundary of Lebesgue measure zero.
Then for any integer $k$, $1\leq k\leq\lfloor\frac12d-1\rfloor$,
we have convergence of moments:
\begin{align}\label{MOMENTCORRES}
\lim_{T\to\infty}\:\frac1{\#(\widehat\N^d\cap T\scrD)}
\sum_{\veca\in\widehat\N^d\cap T\scrD}
\Bigl(\frac{f(\veca)}{(a_1\cdots a_d)^{1/(d-1)}}\Bigr)^k
=M_{d,k}.
\end{align}
If furthermore $\overline\scrD\subset\R_{>0}^d$, then
\eqref{MOMENTCORRES} holds for all $1\leq k\leq d-2$. %
\end{cor}
For $d=3$ and $k=1$ the limit relation \eqref{MOMENTCORRES} 
in fact holds without the extra assumption $\overline\scrD\subset\R_{>0}^d$;
this follows from Ustinov
\cite[Thm.\ 1]{aU2009}.
For $d\geq4$ and $k=1$, \eqref{MOMENTCORRES} was proved in
\cite{iAmHaH2011}.

\vspace{5pt}

Finally let us turn to a %
slightly different question:
What can be said about the limit distribution of
Frobenius numbers for \textit{$d$ large?}
Let 
$\rho_{d-1}$
be the 
absolute inhomogeneous minimum of $\Delta$,
viz.\
\begin{align}
\rho_{d-1}
=\inf\bigl\{\rho(L)\col L\in X_{d-1}\bigr\}.
\end{align}
Using \eqref{PSIFORMULA} and the fact that $\Psi_d$ is continuous
(\cite[Lemma 7]{jM2009}),
one easily shows that
\begin{align}
\Psi_d(R)=1\:\text{ for }\:0\leq R\leq\rho_{d-1};\qquad
\text{and }\qquad
\Psi_d(R)<1\:\text{ for }\:R>\rho_{d-1},
\end{align}
i.e.\ the limit distribution described by $\Psi_d(R)$ has support exactly in
the interval $[\rho_{d-1},\infty)$.
In fact $\rho_{d-1}$ is not only a lower bound for the support of the
limit distribution, but a lower bound on the normalized Frobenius number
for \textit{any} input vector; we have
\begin{align}
\frac{f(\veca)}{(a_1\cdots a_d)^{1/(d-1)}}\geq\rho_{d-1},\qquad\forall 
\veca\in\widehat\NN^d,
\end{align}
cf.\ Aliev and Gruber
\cite[Thm.\ 1.1(i)]{iApG2007} as well as 
R\"odseth \cite{oR90}.
It was noted in   %
\cite[(7)]{iApG2007}
that
\begin{align}\label{AGLOWBOUND}
\rho_{d-1}>(d-1)!^{\frac1{d-1}}.
\end{align}
On the other hand
the number $\rho_{d-1}$ is quite near $(d-1)!^{\frac1{d-1}}$ for $d$ large:
It follows from a bound by Rogers
on lattice coverings by general convex bodies,
\cite{cR59},
refined by Gritzmann
\cite{pG85}
in the case of convex bodies satisfying a mild symmetry condition
(cf.\ also \cite[Sec.\ 9]{rDvF2004}, and use the fact that $\Delta$ can be mapped to a regular
$(d-1)$-simplex by a volume preserving linear map), that
\begin{align}
\rho_{d-1}\leq
(d-1)!^{\frac1{d-1}}\biggl(1+O\biggl(\frac{\log d}d\biggr)\biggr)
\qquad\text{as }\: d\to\infty.
\end{align}

When computing the Frobenius numbers
for modest $d$ and several random large vectors $\veca$, 
one notes that the normalized values
$\frac{f(\veca)}{(a_1\cdots a_d)^{1/(d-1)}}$
most often do not exceed 
the experimental value for the  %
lower bound $\rho_{d-1}$ by more than a constant factor $<2$.
This is seen in Figure \ref{FIGURE} above in the cases $d=3,4,5,6$;
the same phenomenon was also noted in 
\cite[Sec.\ 5 (esp.\ Fig.\ 17)]{dBjHaNsW2005}
for $d=4$ and $d=8$.
The following result shows that this behavior continues as $d\to\infty$;
indeed,
for $d$ large,
the distribution described by $\Psi_d(R)$
has almost all of its mass concentrated 
in the interval
between $(d-1)!^{\frac1{d-1}}$ and $1.757\cdot(d-1)!^{\frac1{d-1}}$.

\begin{thm}\label{SECONDMAINTHM}
Let $\eta_0=0.756\ldots$ be the unique real root of 
$e\log \eta+\eta=0$.
Then for any $\alpha>1+\eta_0$ we have
\begin{align}\label{SECONDMAINTHMRES}
\Psi_d\bigl(\alpha(d-1)!^{\frac1{d-1}}\bigr)\to 0 %
\qquad\text{as }\: d\to\infty,
\end{align}
in fact with an exponential rate.
\end{thm}
In particular, combining Theorem \ref{SECONDMAINTHM} with
Theorem \ref{JENSTHM} and \eqref{AGLOWBOUND},
it follows that for large $d$,
the normalized Frobenius number $\frac{f(\veca)}{(a_1\cdots a_d)^{1/(d-1)}}$
is very likely to lie
between $(d-1)!^{\frac1{d-1}}$ and 
$1.757\cdot(d-1)!^{\frac1{d-1}}$.
In precise terms, %
we have for any fixed $\alpha>\eta_0$:
\begin{align}\notag
\lim_{d\to\infty}
\liminf_{T\to\infty}
\frac1{\#(\widehat\N^d\cap[0,T]^d)}\#\biggl\{\veca\in\widehat\N^d\cap[0,T]^d 
\col(d-1)!^{\frac1{d-1}}<\frac{f(\veca)}{(a_1\cdots a_d)^{1/(d-1)}}
\hspace{50pt}
\\\label{LIMLIMINFSTATEMENT}
<\alpha(d-1)!^{\frac1{d-1}}\biggr\}=1.
\end{align}

Theorem \ref{SECONDMAINTHM} follows from a modification of a
general bound by Rogers on 
lattice coverings of space with convex bodies
\cite{cR58},
further improved by Schmidt 
\cite{wS59}.
We carry this out in Section~\ref{LATTICECOVSEC} below.

\begin{remark}
It is an interesting question whether the bound on 
$\alpha$ in Theorem \ref{SECONDMAINTHM} can be further improved.
Could it be that the limit distribution of Frobenius numbers
in fact
concentrates near $(d-1)!^{\frac1{d-1}}$ as $d\to\infty$,
in the sense that \eqref{SECONDMAINTHMRES} holds for \textit{all} $\alpha>1$?

It is also an interesting task to try prove a good \textit{uniform}
bound on $\Psi_d(R)$ valid for all large $d$ and $R$,
uniting Theorem  \ref{SECONDMAINTHM}
and the fact that $\Psi_d(R)\ll_d R^{-(d-1)}$ as $R\to\infty$.
Even more generally we may ask for a good uniform bound on $P_d(T,R)$
valid for all large $d$, $T$, $R$.
\end{remark}

\subsection*{Acknowledgements}
I am grateful to Jens Marklof for inspiring and helpful discussions.

\section{The asymptotic behavior of $\Psi_d(R)$ as $R\to\infty$}
\label{PSIDASYMPTSEC}

In this section we will prove Theorem \ref*{PSIDASYMPTTHM}.

\subsection{Preliminaries}

Let us write $n=d-1$. Recall that $\Delta$ denotes the standard 
$n$-dimensional simplex defined in \eqref{DELTADEF}.
Given $L\in X_n$ and $\rho>0$, we have
$L+\rho\Delta=\R^n$ if and only if $\veczeta-\rho\Delta$ has non-empty
intersection with $L$ for each $\veczeta\in\R^n$.
Thus, since $L=-L$:
\begin{align}\label{RHOFORMULA}
\rho(L)=\sup\{\rho>0\col\text{there is $\veczeta\in\R^{n}$ such that
$L\cap(\rho\Delta-\veczeta)=\emptyset$}\}.
\end{align}
It follows that the formula for $\Psi_d(R)$, \eqref{PSIFORMULA},
may be rewritten as
\begin{align}\label{PSIFORMULAREWR}
\Psi_d(R)=\mu_n\bigl(\bigl\{L\in X_n\col
\text{there is $\veczeta\in\R^n$ such that
$L\cap(R\Delta-\veczeta)=\emptyset$}\bigr\}.
\end{align}

Let us write $G=G^{(n)}=\SL(n,\R)$ and $\Gamma=\Gamma^{(n)}=\SL(n,\Z)$.
For any $M\in G$, $\Z^nM$ is an $n$-dimensional lattice of covolume one,
and this gives an identification of the space $X_n$ with the
homogeneous space $\Gamma\backslash G$.
Note that $\mu_n$ is the measure on $X_n$ coming from Haar measure on $G$,
normalized to be a probability measure; 
we write $\mu_n$ also for the corresponding Haar measure on $G$.
Let $A=A^{(n)}$ be the subgroup of $G$ consisting of
diagonal matrices with positive entries
\begin{align}\label{ADEF}
\aa(a)=\begin{pmatrix} a_1 & & \\ & \ddots & \\ & & a_n \end{pmatrix}
\in G, \qquad a_j>0,
\end{align}
and let $N=N^{(n)}$ be the subgroup of upper triangular matrices
\begin{align}\label{NUDEF}
\nn(u)=\begin{pmatrix} 1 & u_{12} & \cdots & u_{1n}
\\ & \ddots & \ddots & \vdots 
\\ & & \ddots & u_{n-1,n} 
\\ & & & 1
\end{pmatrix} \in G.
\end{align}
Every element $M\in G$ has a unique Iwasawa decomposition
\begin{align}\label{IWASAWA}
M=\nn(u)\aa(a)\kk,
\end{align}
with $\kk\in \SO(n)$.
We set 
\begin{align}
\mathcal{F}_N=\bigl\{u \col u_{jk} \in (-\sfrac 12,\sfrac 12], \:
1\leq j<k\leq n\bigr\};
\end{align}
then $\{\nn(u) \col u\in \mathcal{F}_N\}$
is a fundamental region for $(\Gamma\cap N)\backslash N$.
We define the following Siegel set:
\begin{align}\label{SIDEF}
\Si_n:=\Bigl\{\nn(u) \aa(a) \kk\in G \col u\in \mathcal{F}_N,\:
0<a_{j+1} \leq \sfrac{2}{\sqrt 3}a_j \: (j=1,\ldots,n-1), 
\: \kk \in \SO(n) \Bigr\}.
\end{align}
It is known that $\Si_n$ contains a fundamental region for
$X_n=\Gamma\backslash G$, and on the other hand
$\Si_n$ is contained in a finite union of fundamental regions for $X_n$
(\cite{aB69}).

\begin{lem}\label{LEM0}
If $R>0$ and $M=\nn(u)\aa(a)\kk\in\Si_n$ satisfy
$\Z^nM\cap(R\Delta-\veczeta)=\emptyset$ for some
$\veczeta\in\R^n$, then $a_1\gg_d R$.
\end{lem}
\begin{proof}
Note that $R\Delta$ contains a ball of radius $\gg_d R$.
Now the lemma follows from 
\cite[Lemma 2.1]{aS2010m}.
\end{proof}

Alternatively, Lemma \ref{LEM0} follows from 
Jarnik's inequalities
(cf., e.g., %
\cite[p.\ 99]{pGcL87})
together with the fact that $a_1\asymp_d\lambda_n$,
where $\lambda_n$ is the last successive mimimum of the lattice 
$\Z^nM$ (cf.\ \eqref{SUCCMINDEF} below).

Let us remark that using the above lemma together with \eqref{PSIFORMULAREWR}
and the bound
\begin{align}
\mu_n\bigl(\bigl\{M\in\Si_n\col a_1>A\bigr\}\bigr)\ll_d A^{-n},
\qquad\forall A>0
\end{align}
(cf.\ the proof of \cite[Lemma 2.4]{aS2010m}),
we immediately deduce the upper bound 
\begin{align}
\Psi_d(R)\ll_d R^{-n}
\end{align}
which was proved by Li 
\cite[Thm.\ 1.2]{hL2010}
in a different (but closely related) way.

We next recall the parametrization of $G=G^{(n)}$ by
$\R_{>0}\times\S_1^{n-1}\times\R^{n-1}\times G^{(n-1)}$
introduced in
\cite[(2.9)--(2.11)]{aS2010m}.
Let us fix a function $f$ (smooth except possibly at one point, say)
$\S^{n-1}_1\to\SO(n)$ such that
$\vece_1 f(\vecv)=\vecv$ for all $\vecv\in \S_1^{n-1}$
(where $\vece_1=(1,0,\ldots,0)$).
Given $M=\nn(u)\aa(a)\kk\in G$,
the matrices $\nn(u)$, $\aa(a)$ and $\kk$ can be split uniquely as
\begin{align}\label{NAKSPLITDEF}
\nn(u)=\matr 1\vecu{\trans\bn}{\nn(\tu)}; \qquad
\aa(a)=\matr{a_1}\bn{\trans\bn}{a_1^{-\frac 1{n-1}}\aa(\ta)}; \qquad
\kk=\matr 1\bn{\trans\bn}{\tkk} f(\vecv)
\end{align}
where $\vecu\in\R^{n-1}$, $\nn(\tu)\in N^{(n-1)}$,
$a_1>0$, $\aa(\ta)\in A^{(n-1)}$ and
$\tkk\in\SO(n-1)$, $\vecv\in\S_1^{n-1}$.
We set
\begin{align}\label{TMDEF}
\tM=\nn(\tu)\aa(\ta)\tkk\in G^{(n-1)}.
\end{align}
In this way we get a bijection between $G$ and
$\R_{>0}\times\S_1^{n-1}\times\R^{n-1}\times G^{(n-1)}$;
we write $M=[a_1,\vecv,\vecu,\tM]$ for the element in $G$ corresponding to
the 4-tuple
$\langle a_1,\vecv,\vecu,\tM\rangle\in
\R_{>0}\times\S_1^{n-1}\times\R^{n-1}\times G^{(n-1)}$.
The Haar measure $\mu_n$ takes the following form
in the parametrization $M=[a_1,\vecv,\vecu,\tM]$:
\begin{align}\label{SLDRSPLITHAAR}
d\mu_n(M)=\zeta(n)^{-1} \,d\mu_{n-1}(\tM)\,
d\vecu \, d\vecv\,\frac{da_1}{a_1^{n+1}},
\end{align}
where $d\vecu$ is standard Lebesgue measure on $\R^{n-1}$ and
$d\vecv$ is the $(n-1)$-dimensional volume measure on $\S_1^{n-1}$
(\cite[(2.12)]{aS2010m}).
Note that all of the above claims are valid also for $n=2$,
with the natural interpretation that $\Si_1=\SL(1,\R)=\{1\}$ 
with $\mu_1(\{1\})=1$.

\subsection{On the intersection of $\Delta$ and a hyperplane 
orthogonal to $\vecv$}
Given $M=[a_1,\vecv,\vecu,\tM]$, the points in the lattice
$\Z^nM$ are given by the formula
\begin{align}\label{LATTICEINPARAM}
(k,\vecm)M
=ka_1\vecv+a_1^{-\frac1{n-1}}\bigl(0,k\vecu\aa(\ta)\tkk+\vecm\tM\bigr)
f(\vecv)
\qquad (\forall k\in\Z,\: \vecm\in\Z^{n-1}).
\end{align}
In particular $\Z^nM$ is contained in the union of the
(parallel) hyperplanes $ka_1\vecv+\vecv^\perp$:
\begin{align}\label{LATTICECONTAINEMENT}
\Z^nM\subset\bigcup_{k\in\Z} \bigl(ka_1\vecv+\vecv^\perp\bigr).
\end{align}
Note that for each $k$, the ($n-1$)-dimensional affine lattice
$\Z^nM\cap(ka_1\vecv+\vecv^\perp)$ has covolume $a_1^{-1}$
inside $ka_1\vecv+\vecv^\perp$.
Hence if $a_1$ is large then this point set %
typically covers $ka_1\vecv+\vecv^\perp$ well in the sense that
the maximal distance from $\Z^nM\cap(ka_1\vecv+\vecv^\perp)$ to any point
in $ka_1\vecv+\vecv^\perp$ is small.

Given $\vecv=(v_1,\ldots,v_n)\in\S_1^{n-1}$ we let 
$P_\vecv:\R^{n}\mapsto\R^{n}$ be orthogonal projection
onto the line $\R\vecv$, viz.
\begin{align}
P_\vecv(\vecx):=(\vecx\cdot\vecv)\vecv.
\end{align}
Note that $P_\vecv(\Delta)$ is a closed line segment;
let us denote by $\ell(\vecv)$ the length of this line segment.
In other words, $\ell(\vecv)$ is the width of $\Delta$ in the direction
$\vecv$.
Since $\Delta$ is the convex hull of 
$\{\bn,\vece_1,\vece_2,\ldots,\vece_{n}\}$,
where $\vece_j$ is the $j$th standard basis vector of $\R^n$,
$P_\vecv(\Delta)$ is the convex hull of
$\{P_\vecv(\bn),P_\vecv(\vece_1),\ldots,P_\vecv(\vece_{n})\}$,
and here $P_\vecv(\bn)=\bn$ and $P_\vecv(\vece_j)=v_j\vecv$.
Hence
\begin{align}
\ell(\vecv)=
\ell_+(\vecv)-\ell_-(\vecv),
\end{align}
where 
\begin{align}\label{ELLPMDEF}
\ell_+(\vecv):=\max(0,v_1,\ldots,v_n);\qquad
\ell_-(\vecv):=\min(0,v_1,\ldots,v_n).
\end{align}
In particular $\frac1{\sqrt{n}}\leq\ell(\vecv)\leq\sqrt2$.
\begin{lem}\label{LEM1}
If $R>0$, $M=[a_1,\vecv,\vecu,\tM]$ and $a_1>\ell(\vecv)R$, then there exists
$\veczeta\in\R^n$ such that $\Z^nM\cap(R\Delta-\veczeta)=\emptyset$.
\end{lem}
\begin{proof}
Because of \eqref{LATTICECONTAINEMENT}, 
$\Z^nM\cap(R\Delta-\veczeta)=\emptyset$ certainly holds
whenever $R\Delta-\veczeta$ lies completely inside the open
strip contained between the two parallel hyperplanes $\vecv^\perp$ and
$a_1\vecv+\vecv^\perp$, and 
this holds if and only if $P_\vecv(R\Delta-\veczeta)\subset
\{t\vecv\col 0<t<a_1\}$. There exist vectors $\veczeta$
satisfying the last inclusion if and only if $\ell(\vecv)R<a_1$.
\end{proof}

We next seek to obtain restrictions on those lattices
$\Z^nM$ with 
$M=[a_1,\vecv,\vecu,\tM]$ and $a_1\leq \ell(\vecv)R$ 
which still satisfy
$\Z^nM\cap(R\Delta-\veczeta)=\emptyset$ for some $\veczeta\in\R^n$.
We first prove the following simple geometric fact.
\begin{lem}\label{LEM2}
For any $\vecv\in\S_1^{n-1}$ and $x\in\R$,
the hyperplane $x\vecv+\vecv^\perp$
intersects $\Delta$ if and only if 
$x\in[\ell_-(\vecv),\ell_+(\vecv)]$,
and furthermore when this happens,
$(x\vecv+\vecv^\perp)\cap\Delta$ contains an $(n-1)$-dimensional ball
of radius %
$(2\sqrt n+n)^{-1}\min\bigl(x-\ell_-(\vecv),\ell_+(\vecv)-x\bigr).$
\end{lem}
\begin{proof}
The first statement follows since
$x\vecv+\vecv^\perp$ intersects $\Delta$ if and only if
$x\vecv\in P_\vecv(\Delta)$, and
$P_\vecv(\Delta)=\{t\vecv\col \ell_-(\vecv)\leq t\leq\ell_+(\vecv)\}$.

To prove the second statement we will prove the stronger fact
that if $x\in[\ell_-(\vecv),\ell_+(\vecv)]$ then
there is some $\vecy\in x\vecv+\vecv^\perp$
such that $\vecy+\scrB_r^{n}\subset\Delta$,
where
\begin{align}\label{LEM2PF1}
r:=(2\sqrt n+n)^{-1}\min\bigl(x-\ell_-(\vecv),\ell_+(\vecv)-x\bigr),
\end{align}
and where $\scrB_r^n$ denotes the closed $n$-dimensional ball of radius $r$
centered at $\bn$ (thus $\vecy+\scrB_r^n$ is the ball of radius $r$ 
centered at $\vecy$).

For an arbitrary point $\vecy=(y_1,\ldots,y_n)\in\R^n$ we 
note that $\vecy+\scrB_r^{n}\subset\Delta$ holds if and only if
$y_1,\ldots,y_{n}\geq r$
and $y_1+\ldots+y_{n}\leq1-\sqrt nr$,
which is equivalent to saying that 
$(\sqrt n+n)r\leq1$ and $\vecy-r\vece\in (1-(\sqrt n+n)r)\Delta$.
The condition $(\sqrt n+n)r\leq1$ is clearly fulfilled
for our $r$,
since %
$\min\bigl(x-\ell_-(\vecv),\ell_+(\vecv)-x\bigr)\leq
\sfrac12\ell(\vecv)
\leq2^{-\frac12}$.

Hence, since $\Delta$ is the convex hull of $\{\bn,\vece_1,\ldots,\vece_n\}$,
it follows that %
there exists a point 
$\vecy\in x\vecv+\vecv^\perp$ with $\vecy+\scrB_r^{n}\subset\Delta$ 
if and only if $x$ lies in the (1-dimensional) convex hull of the
$n+1$ numbers
\begin{align}
r\vecv\cdot\vece  %
\quad\text{and}\quad
r\vecv\cdot\vece %
+\bigl(1-(\sqrt n+n)r\bigr)v_j\quad\text{for}\quad j=1,2,\ldots,n.
\end{align}
Recalling \eqref{ELLPMDEF} we see that this holds if and only if
$x\in[\alpha_-,\alpha_+]$, where
\begin{align}
\alpha_\pm:=r\vecv\cdot\vece+\bigl(1-(\sqrt n+n)r\bigr)\ell_\pm(\vecv)
\end{align}
However
\begin{align}
\bigl|\alpha_\pm-\ell_\pm(\vecv)\bigr|
\leq r|\vecv\cdot\vece|+(\sqrt n+n)r|\ell_\pm(\vecv)|
\leq r\bigl(\sqrt n+\sqrt n+n\bigr).
\end{align}
Hence $x\in[\alpha_-,\alpha_+]$ certainly holds whenever
\begin{align}
\ell_-(\vecv)+(2\sqrt n+n)r\leq x\leq\ell_+(\vecv)-(2\sqrt n+n)r,
\end{align}
and this condition is clearly fulfilled for our $r$ in
\eqref{LEM2PF1}.
\end{proof}

\begin{lem}\label{LEM3}
If $R>0$, $M=[a_1,\vecv,\vecu,\tM]\in\Si_{n}$ and $a_1\leq \ell(\vecv)R$,
and if $\Z^{n}M\cap(R\Delta-\veczeta)=\emptyset$ holds for some 
$\veczeta\in\R^{n}$, then
\textup{$\ta_1\gg_d (\ell(\vecv)R-a_1)a_1^{\frac1{n-1}}$} in 
\textup{$\tM=\nn(\tu)\aa(\ta)\tkk\in G^{(n-1)}$.}
\end{lem}
\begin{proof}
Set $X=\ell(\vecv)R-a_1\geq0$.
Since $P_\vecv(R\Delta-\veczeta)$ is a closed line segment 
in $\R\vecv$ of length $\ell(\vecv)R$,
there exists some $k\in\Z$ such that
$ka_1\vecv\in P_\vecv(R\Delta-\veczeta)$ and furthermore
such that $ka_1\vecv$ has distance $\geq\frac12 X$ to both the endpoints of
$P_\vecv(R\Delta-\veczeta)$.
Hence by Lemma \ref{LEM2},
$(ka_1\vecv+\vecv^\perp)\cap(R\Delta-\veczeta)$ contains an $(n-1)$-dimensional
ball $B$ of radius $\gg_d X$.
Now $\Z^nM\cap(R\Delta-\veczeta)=\emptyset$ implies that
the $(n-1)$-dimensional affine lattice
$(ka_1\vecv+\vecv^\perp)\cap\Z^{n}M$ must be disjoint from $B$.
In view of \eqref{LATTICEINPARAM} it follows that the
$(n-1)$-dimensional lattice $a_1^{-\frac1{n-1}}(0,\Z^{n-1}\tM)f(\vecv)
\subset\vecv^\perp$ is disjoint from 
a certain translate of $B$ inside $\vecv^\perp$.
Hence $\Z^{n-1}\tM$ is disjoint from a ball of radius
$\gg_d a_1^{\frac1{n-1}}X$ in $\R^{n-1}$, and so
$\ta_1\gg_d a_1^{\frac1{n-1}}X$ by 
\cite[Lemma 2.1]{aS2010m}.
\end{proof}

\subsection{The main computation}
\label{MAINCOMPSEC}

Recall that by Lemma \ref{LEM0}, if 
$M=\nn(u)\aa(a)\kk\in\Si_{n}$ satisfies
$\Z^{n}M\cap(R\Delta-\veczeta)=\emptyset$ for some
$\veczeta\in\R^{n}$, then $a_1\geq \kappa R$,
where $\kappa>0$ is a constant which only depends on $d$.
We set 
\begin{align}
A:=\kappa R,
\end{align}
and from now on we keep $R>\kappa^{-1}$, so that $A>1$.

We next recall some definitions and facts from 
\cite[Sec.\ 3.2]{jMaS2010b}.
We fix a subset $\S_\pm^{n-1}\subset\S_1^{n-1}\cap\{v_1\geq0\}$
which contains exactly one of the vectors $\vecv$ and $-\vecv$ for 
every $\vecv\in\S_1^{n-1}$.
Let us also fix a (set theoretical, measurable) fundamental region
$\F_{n-1}\subset\Si_{n-1}$ for $\Gamma^{(n-1)}\backslash G^{(n-1)}$.
We set
(cf.\ \cite[(3.15), (3.18)]{jMaS2010b})
\begin{align}\label{FGDEF}
&
\FG_A:=\Bigl\{[a_1,\vecv,\vecu,\tM]\in G \col a_1>A,\:
\vecv\in\S_\pm^{n-1},\:
\vecu\in(-\sfrac12,\sfrac12]^{n-1},\:
\tM\in\F_{n-1}\Bigr\}
\end{align}
and
\begin{align}\label{SIDPDEF}
\Si_n':=\Bigl\{[a_1,\vecv,\vecu,\tM]\in \Si_n\col\vecv\in\S_\pm^{n-1}\Bigr\}.
\end{align}

\begin{lem}\label{LEM3P4FROMLORENTZFOUR}
There exists a (set-theoretical, measurable) fundametal region
$\F_n\subset\Si_n'$ for $X_n=\Gamma\backslash G$ 
and a (measurable) subset $\FC\subset\Si_n'\cup\FG_A$, such that
\begin{align}
&\FG_A\setminus\FC
\:\:\subset\:\: \bigl\{M\in \F_n\col a_1>A\bigr\}
\:\:\subset\:\:\FG_A\cup\FC
\end{align}
and 
$\mu_n(\FC)\ll_d A^{-2n}$ if $n\geq3$, while $\FC=\emptyset$
if $n=2$.
\end{lem}
\begin{proof}
For $n\geq3$ this follows from
\cite[Lemma 3.4]{jMaS2010b},
together with the computation in 
\cite[(3.23), (3.24)]{jMaS2010b}.
In the remaining case $n=2$ we use the well-known fact that
a fundamental region for $X_2=\Gamma^{(2)}\backslash G^{(2)}$ 
is provided by
\begin{align}
\F_2:=\bigl\{\nn(u)\aa(a)f(\vecv)\in G^{(2)}
\col u+a_1^2i\in\F_\HH,\:\vecv\in\S_\pm^1\bigr\},
\end{align}
where $\F_\HH$ is the usual fundamental region for the action of
$\Gamma^{(2)}$ on the upper half-plane $\HH=\{z=x+iy\in\CC\col y>0\}$, viz.\ 
\begin{align}
\F_\HH:=\Bigl\{z=x+iy\in\HH\col -\sfrac12<x\leq\sfrac12,\:|z|\geq1,\:
(x<0\Rightarrow |z|>1)\Bigr\}.
\end{align}
In particular for this choice of $\F_2$ we have
$\F_2\subset\Si_2'$ and $\{M\in \F_2\col a_1>A\}=\FG_A$, since $A>1$.
\end{proof}

It follows from Lemma \ref{LEM3P4FROMLORENTZFOUR} and
\eqref{PSIFORMULAREWR} that
\begin{align}
\Psi_d(R)=\int_{\FG_A}I\Bigl(\exists\veczeta\in\R^{n}:\:
\Z^{n}M\cap(R\Delta-\veczeta)=\emptyset\Bigr)\,d\mu_n(M)
+O\bigl(\mu_n(\FC)\bigr),
\end{align}
where the error term is $\ll_d A^{-2n}\ll_d R^{-2n}$ if $n\geq3$,
while if $n=2$ then the error term vanishes.
Hence, using \eqref{FGDEF} and \eqref{SLDRSPLITHAAR}, we obtain
\begin{align}\notag
\Psi_d(R)=\frac1{\zeta(n)}\int_A^\infty\int_\HSm
\int_{(-\frac12,\frac12)^{n-1}}\int_{\F_{n-1}}
I\Bigl(\exists\veczeta\in\R^{n}:\:
\Z^{n}[a_1,\vecv,\vecu,\tM]\cap(R\Delta-\veczeta)=\emptyset\Bigr)
\\\label{PSIDASYMPTSTEP1}
\times d\mu_{n-1}(\tM)\,d\vecu\,d\vecv\,\frac{da_1}{a_1^{n+1}}
+O_d\bigl(I(n\geq3)\cdot R^{-2n}\bigr).
\end{align}
Here it follows from Lemma \ref{LEM1} that the integral is 
\begin{align}\label{LOWBOUND}
\geq\frac1{\zeta(n)}\int_{\HSm}\int_{\ell(\vecv)R}^\infty
\frac{da_1}{a_1^{n+1}}\,d\vecv
=\frac{R^{-n}}{n\zeta(n)}\int_{\HSm}\ell(\vecv)^{-n}\,d\vecv.
\end{align}
(Note here that by Lemma \ref{LEM1} and our definition of $A$
we have $A\leq \ell(\vecv)R$ for all $\vecv\in\S_1^{n-1}$.)
On the other hand it follows from Lemma \ref{LEM3} that 
there is a constant $\kappa'>0$ which only depends on $d$ such that
difference between the integral in \eqref{PSIDASYMPTSTEP1} and the
right hand side of \eqref{LOWBOUND} is
\begin{align}
\leq\frac1{\zeta(n)}\int_\HSm\int_A^{\ell(\vecv)R}
\mu_{n-1}\Bigl(\Bigl\{\tM %
\in\F_{n-1}\col
\ta_1\geq\kappa'(\ell(\vecv)R-a_1)a_1^{\frac1{n-1}}
\Bigr\}\Bigr)\,\frac{da_1}{a_1^{n+1}}\,d\vecv.
\end{align}
Here $A=\kappa R$; hence $R\ll_d a_1\ll_d R$ throughout the integral, 
and we get, with a new constant $\kappa''>0$ which only depends on $d$:
\begin{align}\notag
&\ll_d R^{-(n+1)}\int_\HSm\int_{\kappa R}^{\ell(\vecv)R}
\mu_{n-1}\Bigl(\Bigl\{\tM\in\F_{n-1}\col
\ta_1\geq\kappa''(\ell(\vecv)R-a_1)R^{\frac1{n-1}}
\Bigr\}\Bigr)\,da_1\,d\vecv
\\\notag
&\leq R^{-(n+1)}\int_\HSm\int_0^{\ell(\vecv) R}
\mu_{n-1}\Bigl(\Bigl\{\tM\in\F_{n-1}\col\ta_1\geq\kappa'' tR^{\frac1{n-1}}
\Bigr\}\Bigr)\,dt\,d\vecv.
\\\label{COMPUTATION}
&\ll_d R^{-(n+1)}\int_0^{\sqrt2 R}
\mu_{n-1}\Bigl(\Bigl\{\tM\in\F_{n-1}\col\ta_1\geq\kappa'' tR^{\frac1{n-1}}
\Bigr\}\Bigr)\,dt.
\end{align}
Now if $n\geq3$ then by a computation as in the proof of 
\cite[Lemma 2.4]{aS2010m}
we get
\begin{align}
\ll_d R^{-(n+1)}\int_0^{\sqrt2 R}
\bigl(1+tR^{\frac1{n-1}}\bigr)^{-(n-1)}\,dt
\ll_d R^{-n-1-\frac1{n-1}}.
\end{align}
On the other hand if $n=2$ then $\F_{n-1}=\{1\}$ and hence
the last line of \eqref{COMPUTATION}
equals $R^{-3}\cdot\min(\sqrt2 R,{\kappa''}^{-1}R^{-1})$, which is
$\ll R^{-4}$.  %
Hence we conclude:
\begin{align}\label{PSIDASYMPT1}
\Psi_d(R)=
\frac{R^{-n}}{n\zeta(n)}\int_{\HSm}\ell(\vecv)^{-n}\,d\vecv
+O_d\bigl(R^{-n-1-\frac1{n-1}}\bigr).
\end{align}
Now to prove the asymptotic formula for $\Psi_d(R)$
stated in Theorem \ref{PSIDASYMPTTHM},
it only remains to compute the integral
$\int_{\HSm}\ell(\vecv)^{-n}\,d\vecv$.

\subsection{Computing the constant in the main term}
\label{CONSTANTSEC}

\begin{lem}\label{CONSTANTFORMULALEM}
For every $n\geq2$ we have
\begin{align}
\int_{\HSm}\ell(\vecv)^{-n}\,d\vecv=\frac{n(n+1)}2.
\end{align}
\end{lem}

\begin{proof}
Set
\begin{align}
K=\bigl\{r\vecv\col\vecv\in\S_1^{n-1},\:
0\leq r\leq\ell(\vecv)^{-1}\bigr\}\subset\R^n;
\end{align}
then clearly
\begin{align}\label{CONSTANTFORMULALEMPF1}
\int_{\HSm}\ell(\vecv)^{-n}\,d\vecv=
\frac12\int_{\S_1^{n-1}}\ell(\vecv)^{-n}\,d\vecv=
\frac n2\vol(K).
\end{align}
But for any $\vecx=r\vecv$ with $r>0$ and $\vecv\in\S_1^{n-1}$ we have
\begin{align}
\ell(\vecv)
=\|\vecx\|^{-1}\bigl(\max(0,x_1,\ldots,x_n)-\min(0,x_1,\ldots,x_n)\bigr),
\end{align}
so that $r\leq\ell(\vecv)^{-1}$ holds if and only if
$\max(0,x_1,\ldots,x_n)-\min(0,x_1,\ldots,x_n)\leq1$.
In other words,
\begin{align}
K=\bigl\{\vecx\in[-1,1]^n\col |x_j-x_k|\leq1,\:\forall j,k\bigr\}.
\end{align}
Hence by easy symmetry considerations we have
\begin{align}\notag
\vol(K)=\vol\bigl(K\cap[0,1]^n\bigr)+\vol\bigl(K\cap[-1,0]^n\bigr)
\hspace{200pt}
\\\notag
+n(n-1)\vol\bigl(\bigl\{\vecx\in K\col
x_1<0<x_2\text{ and } x_1<x_j<x_2\text{ for }j=3,\ldots,n\bigr\}\bigr)
\\\label{CONSTANTFORMULALEMPF2}
=2+n(n-1)\int_{-1}^0\int_0^{1+x_1}(x_2-x_1)^{n-2}\,dx_2\,dx_1
=n+1.
\hspace{117pt}
\end{align}
The lemma follows from \eqref{CONSTANTFORMULALEMPF1} and
\eqref{CONSTANTFORMULALEMPF2}.
\end{proof}

\subsection{Bound from below}

Finally we will prove the lower bound 
\eqref{PSIDASYMPTTHMRES2}
in Theorem \ref{PSIDASYMPTTHM}.

The key step is the following lemma, which says that for ``good'' directions 
$\vecv=(v_1,\ldots,v_n)\in\S_1^{n-1}$,
we may weaken the restriction $a_1>\ell(\vecv)R$ in
Lemma \ref{LEM1} by a small but uniform amount,
and still be sure to have $\Z^nM\cap(R\Delta-\veczeta)=\emptyset$
for some $\veczeta\in\Z^n$.

\begin{lem}\label{BETTERRESTRLEM}
Let $c$ be a fixed number in the interval $(0,n^{-\frac12})$,
and set 
\begin{align}\label{BETTERRESTRLEMCP}
c'=%
(n-1)!^{\frac1{n-1}}c^{\frac n{n-1}}.
\end{align}
Then for any 
$R\geq(2c'\sqrt n)^{1-\frac1n}$ and any
$M=[a_1,\vecv,\vecu,\tM]$ with
$a_1>\ell(\vecv)R-c'R^{-\frac1{n-1}}$ and $v_j>c$ %
($\forall j$),
there exists $\veczeta\in\R^n$ such that
$\Z^nM\cap(R\Delta-\veczeta)=\emptyset$.
\end{lem}
\begin{proof}
Let $R$ and $M=[a_1,\vecv,\vecu,\tM]$ satisfy the given assymptions.
If $a_1>\ell(\vecv)R$ then the desired statement is in 
Lemma \ref{LEM1}; hence from now on we may assume $a_1\leq\ell(\vecv)R$.
We will choose
\begin{align}
\veczeta=c'R^{-\frac1{n-1}}\vecv+\vecw
\end{align}
for some $\vecw\in\vecv^\perp$ which will be fixed at the end of the proof.
Then for every $\vecx\in R\Delta-\veczeta$ we have
\begin{align}\label{BETTERRESTRLEMPF2}
\vecx\cdot\vecv\leq \ell_+(\vecv)R-\veczeta\cdot\vecv
=\ell(\vecv)R-c'R^{-\frac1{n-1}}
\end{align}
and
\begin{align}\label{BETTERRESTRLEMPF3}
\vecx\cdot\vecv\geq-\veczeta\cdot\vecv=-c'R^{-\frac1{n-1}}
\geq-\bigl(\ell(\vecv)R-c'R^{-\frac1{n-1}}\bigr),
\end{align}
where we used the assumption $R\geq(2c'\sqrt n)^{1-\frac1n}$ in the last step.
Using \eqref{BETTERRESTRLEMPF2}, \eqref{BETTERRESTRLEMPF3} and
$a_1>\ell(\vecv)R-c'R^{-\frac1{n-1}}$ we conclude that
\begin{align}\label{BETTERRESTRLEMPF1}
(R\Delta-\veczeta)\cap(ka_1\vecv+\vecv^\perp)=\emptyset,
\qquad\forall k\in\Z\setminus\{0\}.
\end{align}
Hence, using also \eqref{LATTICECONTAINEMENT}, it follows that
\begin{align}
(R\Delta-\veczeta)\cap\Z^nM
=(R\Delta-\veczeta)\cap L_{M,\vecv},
\end{align}
where $L_{M,\vecv}$ is the $(n-1)$-dimensional lattice
$L_{M,\vecv}=\Z^nM\cap\vecv^\perp$.
Recall that $L_{M,\vecv}$ has covolume $a_1^{-1}$ in $\vecv^\perp$.
Using also $R\Delta\subset\R_{\geq0}^n$ and 
$\veczeta=c'R^{-\frac1{n-1}}\vecv+\vecw$, $\vecw\in\vecv^\perp$,
we obtain
\begin{align}\notag
(R\Delta-\veczeta)\cap\Z^nM
&\subset
(\R_{\geq0}^n-c'R^{-\frac1{n-1}}\vecv-\vecw)\cap L_{M,\vecv}
\\\label{BETTERRESTRLEMPF4}
&=\bigl(
((\R_{\geq0}^n-c'R^{-\frac1{n-1}}\vecv)\cap\vecv^\perp)
-\vecw\bigr)\cap L_{M,\vecv}.
\end{align}
Here $(\R_{\geq0}^n-c'R^{-\frac1{n-1}}\vecv)\cap\vecv^\perp$ is a
closed $(n-1)$-dimensional simplex, and a simple computation 
yields for its volume
(cf.\ \cite[(17)]{lFsR2007}, or the simpler computation in 
\cite[Lemma 1]{vA2006}):
\begin{align}
\vol_{n-1}\Bigl((\R_{\geq0}^n-c'R^{-\frac1{n-1}}\vecv)\cap\vecv^\perp\Bigr)
=\frac{\prod_{j=1}^n v_j^{-1}}{(n-1)!}
\bigl(c'R^{-\frac1{n-1}}\bigr)^{n-1}
<R^{-1}.
\end{align}
Here in the last step we used $v_j>c$ ($\forall j$) and 
\eqref{BETTERRESTRLEMCP}.
However the covolume of $L_{M,\vecv}$ in $\vecv^\perp$ is,
since we assumed $a_1\leq\ell(\vecv)R$ from start,
\begin{align}
\vol_{n-1}\bigl(\vecv^\perp/L_{M,\vecv}\bigr)=a_1^{-1}
\geq (\ell(\vecv)R)^{-1}>R^{-1}.
\end{align}
(Indeed $\ell(\vecv)=\ell_+(\vecv)<1$ since all $v_j$ are positive.)
The above shows that the volume of 
$(\R_{\geq0}^n-c'R^{-\frac1{n-1}}\vecv)\cap\vecv^\perp$ is
smaller than the covolume of $L_{M,\vecv}$, and hence 
there is some $\vecw\in\vecv^\perp$ such that the intersection in
\eqref{BETTERRESTRLEMPF4} is empty.
\end{proof}

We now return to the computation in Section \ref{MAINCOMPSEC}.
We will bound the
difference between the integral in \eqref{PSIDASYMPTSTEP1} and the
right hand side of \eqref{LOWBOUND} from \textit{below.}
Fix a constant $c\in(0,n^{-\frac12})$ as in Lemma~\ref{BETTERRESTRLEM},
let $c'>0$ be as in \eqref{BETTERRESTRLEMCP},
and let $\Omega$ be the nonempty, relatively open subset of $\HSm$ 
consisting of
all $\vecv=(v_1,\ldots,v_n)\in\S_1^{n-1}$ with $v_j>c$ ($\forall j$).
It now follows from Lemma~\ref{BETTERRESTRLEM} that,
for any $R\geq(2c'\sqrt n)^{1-\frac1n}$, the
difference between the integral in \eqref{PSIDASYMPTSTEP1} and the
right hand side of \eqref{LOWBOUND} is
\begin{align}
\geq\frac1{\zeta(n)}\int_{\ell(\vecv)R-c'R^{-\frac1{n-1}}}^{\ell(\vecv)R}
\int_\Omega d\vecv\,\frac{da_1}{a_1^{n+1}}
\gg_d R^{-n-1-\frac1{n-1}}.
\end{align}
In particular note that this contribution is asymptotically larger
than the error term %
in \eqref{PSIDASYMPTSTEP1}.
Hence we conclude that there exist constants $c,c'>0$ which only depend on $n$
such that for all $R>c'$,
\begin{align}
\Psi_d(R)>\frac{R^{-n}}{n\zeta(n)}\int_{\HSm}\ell(\vecv)^{-n}\,d\vecv
+cR^{-n-1-\frac1{n-1}}.
\end{align}
In view of Lemma \ref{CONSTANTFORMULALEM} we have thus proved 
\eqref{PSIDASYMPTTHMRES2} in Theorem \ref{PSIDASYMPTTHM}.
Since the asymptotic relation \eqref{PSIDASYMPTTHMRES} 
follows from \eqref{PSIDASYMPT1} and 
Lemma \ref{CONSTANTFORMULALEM}, this concludes the proof of 
Theorem \ref{PSIDASYMPTTHM}.
\hfill$\square\square\square$

\section{\texorpdfstring{Uniform bounds on $\widetilde P_d(T,R)$ and $P_d(T,R)$}{Uniform bounds on tPd(T,R) and Pd(T,R)}}
\label{UNIFSEC}

In this section we will prove Theorem \ref{UNIFTHM} and
Corollary \ref{UNIFTHMCOR}.

Let us first note that the claim \eqref{UNIFTHMRES2} in Theorem \ref{UNIFTHM},
i.e.\ 
\begin{align}\label{UNIFTHMRES2rep}
\widetilde P_d(T,R)=0\qquad\text{whenever }\: R\geq
\kappa_\scrD^{1-\frac1{d-1}}T^{1-\frac1{d-1}}
\end{align}
where
\begin{align}\label{BDDEF}
\kappa_\scrD:=\sup_{\vecx\in\scrD}\|\vecx\|,
\end{align}
is a direct consequence of any among several known bounds on the
Frobenius number %
(cf., e.g., \cite{aR2005}).
For example, the classical bound by Schur
(cf.\ \cite{aB42a})
asserts that for any $\veca\in\widehat\N^d$ satisfying 
$a_1\leq a_2\leq\cdots\leq a_d$,
\begin{align}\label{SCHURBOUND}
g(\veca)\leq a_1a_d-a_1-a_d
\qquad(\text{thus}\:
f(\veca)\leq a_1a_d+a_2+\ldots+a_{d-1}<da_1a_d).  %
\end{align}
Using this together with the fact that 
$s(\veca)\geq da_1a_d\|\veca\|^{-1+1/(d-1)}$
for any such $\veca$, we deduce
\begin{align}\label{FSINEQ}
\frac{f(\veca)}{s(\veca)}<\|\veca\|^{1-\frac1{d-1}}.
\end{align}
Here both the left and the right hand sides are invariant
under permutations of the coefficients of $\veca$;
hence \eqref{FSINEQ} in fact holds for \textit{all} $\veca\in\widehat\N^d.$
Finally, \eqref{UNIFTHMRES2rep} follows from \eqref{FSINEQ}.

We next turn to the proof of \eqref{UNIFTHMRES1} in Theorem \ref{UNIFTHM}.
As in the previous section we write $n=d-1$.
Given $\veca\in\widehat\N^d$ we set
\begin{align}
\Lambda_\veca=\Z^d\cap\veca^\perp
=\bigl\{\vecx\in\Z^d\col\veca\cdot\vecx=0\bigr\}.
\end{align}
This is an $n$-dimensional sublattice of $\Z^d$ of determinant 
$\det(\Lambda_\veca)=\|\veca\|$.
(By the determinant, $\det\Lambda$, of a lattice $\Lambda$ of not necessarily
full rank in $\R^d$, we mean the covolume of $\Lambda$ in
$\spn_\R\Lambda$.)
Given any $n$-dimensional %
lattice $\Lambda\subset\R^d$
we write $0<\lambda_1(\Lambda)\leq\cdots\leq\lambda_n(\Lambda)$ for the
Minkowski successive minima of $\Lambda$, i.e.
\begin{align}\label{SUCCMINDEF}
\lambda_j(\Lambda)=\inf\bigl\{r>0\col\dim\spn_\R(\scrB_r^d\cap\Lambda)\geq j\}.
\end{align}
(Recall that $\scrB_r^d$ is the closed $d$-dimensional ball
of radius $r$ centered at $\bn$.)
Then by Aliev and Henk
\cite[(14)]{iAmH2008}\footnote{Note
that ``$\lambda_j$'' in 
\cite{iAmH2008} equals $\|\veca\|^{-\frac1n}\lambda_j(\Lambda_\veca)$ in
our notation.}
(cf.\ also Kannan \cite[Thm.\ 2.5]{rK92})
we have
\begin{align}
\frac{f(\veca)}{s(\veca)}\leq\sfrac12 n\|\veca\|^{-\frac1n}
\lambda_n(\Lambda_\veca).
\end{align}
Note also that we have
$\#(\widehat\N^d\cap T\scrD)\asymp_{d,\scrD}T^d$ uniformly
over all $T>0$ for which $\widehat\N^d\cap T\scrD\neq\emptyset$,
since $\scrD$ is bounded with nonempty interior.
Using these facts together with the fact that
$\Lambda_\veca\neq\Lambda_\vecb$ for all
$\veca\neq\vecb\in\widehat\N^d$ 
(since 
$\spn_\R\Lambda_\veca=\veca^\perp\neq\vecb^\perp=\spn_\R\Lambda_\vecb$),
it follows that
\begin{align}\label{UNIFTHMPF1}
\widetilde P_d(T,R)\ll_{d,\scrD} T^{-d}
\#\Bigl\{\Lambda\in \scrL_n\col\det(\Lambda)\leq \kappa_\scrD T,
\:\lambda_n(\Lambda)>2n^{-1}\det(\Lambda)^{1/n} R\Bigr\},
\end{align}
where $\scrL_n$ is the set of all $n$-dimensional sublattices of 
$\Z^d$.  %

Let us set
\begin{align}
\rho_j(\Lambda):=\lambda_{j+1}(\Lambda)/\lambda_j(\Lambda)
\qquad\text{for }\: j=1,\ldots,n-1.
\end{align}
(Thus $\rho_j(\Lambda)\geq1$ for all $\Lambda$.)
Also, for any $\vecr=(r_1,\ldots,r_{n-1})\in\R_{\geq1}^{n-1}$, we set
\begin{align}
\scrL_n(\vecr):=\Bigl\{\Lambda\in\scrL_n\col
\rho_j(\Lambda)\geq r_j\:(\forall j) %
\Bigr\}.
\end{align}
Now as a special case of Schmidt's
\cite[Thm.\ 5]{wS98},
the number of lattices in $\scrL_n(\vecr)$ with determinant at most $T$
is given by the following asymptotic formula with a precise error term.
Let us write $\rho_j(L)=\lambda_{j+1}(L)/\lambda_j(L)$
also for an $n$-dimensional lattice $L\subset\R^n$,
with $\lambda_1(L)\leq\cdots\leq\lambda_n(L)$ being the
successive minima of $L$.
\begin{thm}\label{SCHMIDTTHM}
(\cite[Thm.\ 5]{wS98})
For any $\vecr\in\R_{\geq1}^{n-1}$ and $T>0$ we have
\begin{align}\notag
\#\Bigl\{\Lambda\in \scrL_n(\vecr)\col\det(\Lambda)\leq T\Bigr\}
=\frac{\pi^{\frac d2}}{2\Gamma(1+\frac d2)}
\Bigl(\prod_{j=2}^n\zeta(j)\Bigr)
\mu_n\bigl(\bigl\{L\in X_n\col\rho_j(L)\geq r_j\:(\forall j)\bigr\}\bigr)
\cdot T^d
\\
+O_d\biggl(\Bigl(\prod_{j=1}^{n-1}r_j^{-(j-\frac1n)(n-j)}\Bigr)
T^{d-\frac1n}\biggr).
\end{align}
Furthermore,
\begin{align}
\mu_n\bigl(\bigl\{L\in X_n\col\rho_j(L)\geq r_j\:(\forall j)\bigr\}\bigr)
\asymp_d\prod_{j=1}^{n-1}r_j^{-j(n-j)}.
\end{align}
\end{thm}
For our argument we will only make use of the upper bound which follows 
from the above theorem, viz.\ 
\begin{align}\label{SCHMIDTTHMRES2}
\#\Bigl\{\Lambda\in \scrL_n(\vecr)\col\det(\Lambda)\leq T\Bigr\}
\ll_d T^{d}\prod_{j=1}^{n-1}r_j^{-j(n-j)}
\biggl(1+T^{-\frac1n}\prod_{j=1}^{n-1}r_j^{\frac1n(n-j)}\biggr).
\end{align}

We will now form a finite %
union of sets $\scrL_n(\vecr)$
which contains the set in the right hand side of \eqref{UNIFTHMPF1}.

For any $n$-dimensional lattice $\Lambda$ we have
\begin{align}
\lambda_n(\Lambda)^n=\prod_{j=1}^n\lambda_j(\Lambda)
\prod_{j=1}^{n-1}\rho_j(\Lambda)^{j}
\asymp_d\det(\Lambda)\prod_{j=1}^{n-1}\rho_j(\Lambda)^{j},
\end{align}
where in the last step we used Minkowski's Second Theorem 
(cf., e.g., \cite[Lectures 3-4]{cS89}).
Hence there exists a constant $c>0$ which only depends 
on $n$ (viz., only on $d$) such that for any 
$n$-dimensional lattice $\Lambda$ and any $R>0$, we have
\begin{align}\label{UNIFTHMPF2}
\lambda_n(\Lambda)>2n^{-1}\det(\Lambda)^{1/n}R
\:\Longrightarrow\:
\prod_{j=1}^{n-1}\rho_j(\Lambda)^j>cR^n.
\end{align}
Note that \eqref{UNIFTHMRES1} is trivial
when $R\ll1$ (since $\widetilde P_d(T,R)\leq1$ always); hence from now on
we may keep $R\geq ec^{-\frac1n}$ without loss of generality.
Set
\begin{align}
B:=\lfloor \log(cR^n)-n\rfloor\in\Z_{\geq0},
\end{align}
and
\begin{align}
\scrR(n,R):=\Bigl\{\vecr=\bigl(e^{b_1},e^{b_2/2},e^{b_3/3},
\ldots,e^{b_{n-1}/(n-1)}\bigr)\col\vecb\in\Z_{\geq0}^{n-1},\:
\sum_{j=1}^{n-1} b_j %
=B  %
\Bigr\}.
\end{align}
Note that if $\Lambda$ is any $n$-dimensional lattice satisfying
$\prod_{j=1}^{n-1}\rho_j(\Lambda)^j>cR^n$, then if we set
$b_j:=\lfloor j\log\rho_j(\Lambda)\rfloor$ we have
\begin{align}
\sum_{j=1}^{n-1}b_j>\sum_{j=1}^{n-1}(j\log\rho_j(\Lambda)-1)
>\log(cR^n)-(n-1)>\log(cR^n)-n\geq B. %
\end{align}
Hence there is a way to decrease some of the $b_j$'s
so as to make $\sum_{j=1}^{n-1}b_j=B$, %
while keeping $\vecb=(b_1,\ldots,b_{n-1})\in\Z_{\geq0}^{n-1}$.
Of course the new vector $\vecb=(b_1,\ldots,b_{n-1})$ still satisfies
$b_j\leq j\log\rho_j(\Lambda)$ for each $j$,
i.e.\ $\rho_j(\Lambda)\geq e^{b_j/j}$.
We have thus proved that for %
any $n$-dimensional lattice $\Lambda$ satisfying
$\prod_{j=1}^{n-1}\rho_j(\Lambda)^j>cR^n$, there exists some
$\vecr\in\scrR(n,R)$ such that
$r_j\leq\rho_j(\Lambda)$ for $j=1,\ldots,n-1$.
This fact together with \eqref{UNIFTHMPF2} imply that
the set in the right hand side of \eqref{UNIFTHMPF1} is contained
in the union of $\scrL_n(\vecr)$ over all $\vecr\in\scrR(n,R)$.
Hence, by \eqref{UNIFTHMPF1}, we have for all $T>0$ with 
$\widehat\N^d\cap T\scrD\neq\emptyset$ 
and all $R\geq ec^{-\frac1n}$,
\begin{align}
\widetilde P_d(T,R)\ll_{d,\scrD} T^{-d}\sum_{\vecr\in\scrR(n,R)}
\#\Bigl\{\Lambda\in\scrL_n(\vecr)\col\det(\Lambda)\leq \kappa_\scrD T\Bigr\}.
\end{align}
Hence, via \eqref{SCHMIDTTHMRES2}, 
\begin{align}\notag
\widetilde P_d(T,R)
\ll_{d,\scrD}
\sum_{\substack{\vecb\in\Z_{\geq0}^{n-1}\\ b_1+\ldots+b_{n-1}=B}}
\exp\Bigl\{-\sum_{j=1}^{n-1}(n-j)b_j\Bigr\}
\hspace{150pt}
\\
+T^{-\frac1n}
\sum_{\substack{\vecb\in\Z_{\geq0}^{n-1}\\ b_1+\ldots+b_{n-1}=B}}
\exp\Bigl\{-\sum_{j=1}^{n-1}(1-(nj)^{-1})(n-j)b_j\Bigr\}.
\end{align}

If $n=2$ then each sum above has exactly one term, and we conclude
\begin{align}
\widetilde P_3(T,R)\ll_{\scrD} R^{-2}+T^{-\frac12}R^{-1}.
\end{align}
If $R<\kappa_\scrD^{\frac12}T^{\frac12}$ then this gives
$\widetilde P_3(T,R)\ll_{\scrD} R^{-2}$.
On the other hand if $R\geq\kappa_\scrD^{\frac12}T^{\frac12}$ then 
$\widetilde P_3(T,R)=0$ by \eqref{UNIFTHMRES2rep}.
Hence the proof of \eqref{UNIFTHMRES1} is complete in the case $n=2$.

We now assume $n\geq3$. 
We set
\begin{align}
\gamma_1(j):=n-j\quad\text{and}\quad
\gamma_2(j)=(1-(nj)^{-1})(n-j)=n+n^{-1}-(j+j^{-1}).
\end{align}
Now for any $\vecb\in\Z_{\geq0}^{n-1}$ with $b_1+\ldots+b_{n-1}=B$
and $b_1+\ldots+b_{n-2}=:s$ we have,
since $\gamma_1(j)$ is a decreasing function of $j$,
\begin{align}
\sum_{j=1}^{n-1}\gamma_1(j)b_j\geq
\gamma_1(n-2)\sum_{j=1}^{n-2}b_j+\gamma_1(n-1)b_{n-1}
=2s+(B-s)=B+s.
\end{align}
Similarly, since also $\gamma_2(j)$ is a decreasing function of $j$
for $j\geq1$,
\begin{align}\notag
\sum_{j=1}^{n-1}\gamma_2(j)b_j
&\geq\gamma_2(n-2)s+\gamma_2(n-1)(B-s)
\\
&=\Bigl(1-\frac1{n(n-1)}\Bigr)B+\Bigl(1-\frac1{(n-1)(n-2)}\Bigr)s.
\end{align}
Note also that for any $s\in\{0,1,\ldots,B\}$ there are exactly
$\binom{s+n-3}{n-3}$ vectors $\vecb\in\Z_{\geq0}^{n-1}$ satisfying
$b_1+\ldots+b_{n-1}=B$ and $b_1+\ldots+b_{n-2}=s$.
Hence
\begin{align}\notag
\widetilde P_d(T,R)
\ll_{d,\scrD}\sum_{s=0}^B\binom{s+n-3}{n-3}e^{-B-s}
+T^{-\frac1n}\sum_{s=0}^B\binom{s+n-3}{n-3}e^{-(1-\frac1{n(n-1)})B
-(1-\frac1{(n-1)(n-2)})s}
\\
\ll_{d,\scrD} e^{-B}+T^{-\frac1n}e^{-(1-\frac1{n(n-1)})B}
\ll_d R^{-n}\bigl(1+T^{-\frac1n}R^{\frac1{n-1}}\bigr).
\end{align}
If $R<\kappa_\scrD^{1-\frac1n}T^{1-\frac1n}$ then this gives
$\widetilde P_d(T,R)\ll_{d,\scrD} R^{-n}$.
On the other hand if $R\geq\kappa_\scrD^{1-\frac1n}T^{1-\frac1n}$ then 
$\widetilde P_d(T,R)=0$ by \eqref{UNIFTHMRES2rep}.
Hence the proof of \eqref{UNIFTHMRES1} is complete.
\hfill $\square$ $\square$ $\square$

\begin{remark}
Note that our proof makes crucial use of the precise error terms which
Schmidt has worked out for the asymptotic formulas in
\cite[Sec.\ 2]{wS98}.
In this vein, note that the proof of the bound
$\widetilde P_d(T,R)\ll_d R^{-2}$
in \cite[Thm.\ 1.1]{iAmH2008}
is correct as it stands only when $T$ 
is sufficiently large in a way which may depend on $R$ (as well as $d$);
this is because the proof in \cite{iAmH2008} uses Schmidt's
\cite[Thm.\ 2]{wS98}
in which the rate of convergence may depend in an unspecified way on the
chosen set ``$\mathscr{D}$'' of lattice similarity classes.

\end{remark}

\subsection{Proof of Corollary \ref*{UNIFTHMCOR}}

Let us first note that \eqref{UNIFTHMCORRES2} is again a
direct consequence of the classical bound by
Schur, \eqref{SCHURBOUND}. Indeed, for any 
$\veca\in\widehat\N^d$ satisfying 
$a_1\leq a_2\leq\cdots\leq a_d$ we have by \eqref{SCHURBOUND}:
\begin{align}
\frac{f(\veca)}{(a_1\cdots a_d)^{\frac1{d-1}}}
<d\cdot\frac{a_1}{(a_1\cdots a_{d-1})^{\frac1{d-1}}}\cdot a_d^{1-\frac1{d-1}}
\leq da_d^{1-\frac1{d-1}}
<d\|\veca\|^{1-\frac1{d-1}},
\end{align}
and this implies \eqref{UNIFTHMCORRES2}.

The following lemma refines 
\cite[Thm.\ 2 and Remark 1]{iAmHaH2011}.
Recall that $n=d-1\geq2$.
Let us write $\|\vecx\|_\infty:=\max(|x_1|,\ldots,|x_n|)$ for the
maximum norm of a vector $\vecx\in\R^n$.

\begin{lem}\label{AGMLEMMA}
For any $T>0$ and $\alpha>0$ we have
\begin{align}\notag
\#\biggl\{\vecx=(x_1,\ldots,x_n)\in\N^n\col \|\vecx\|_\infty \leq T,\:
\frac{\|\vecx\|_\infty}{(x_1\cdots x_n)^{1/n}}>\alpha\biggr\}
\ll_n T^n\alpha^{-n}(\log(2+\alpha))^{n-2}.
\end{align}
\end{lem}
\begin{remark}
For any fixed $\ve>0$ the above bound is in fact %
sharp in the range
$1\leq\alpha\leq T^{1-\frac1n-\ve}$, 
in the sense that the cardinality in the left hand side is
also $\gg_{n,\ve} T^n\alpha^{-n}(\log(2+\alpha))^{n-2}$
uniformly over all $T\geq T_0(n,\ve)$ and all
$1\leq\alpha\leq T^{1-\frac1n-\ve}$.
However we do not need this fact and
we will not prove it here. %
\end{remark}

\begin{proof}[Proof of Lemma \ref{AGMLEMMA}]
It suffices to prove
\begin{align}\label{AGMLEMMAPF1}
\#\biggl\{\vecx\in\N^n\col \sfrac12T<\|\vecx\|_\infty \leq T,\:
\frac{\|\vecx\|_\infty}{(x_1\cdots x_n)^{1/n}}>\alpha\biggr\}
\ll_n T^n\alpha^{-n}(\log(2+\alpha))^{n-2},
\end{align}
since the lemma then follows by dyadic decomposition in the $T$-variable.
Of course we may assume $T\geq1$ since otherwise the set in the left hand
side is empty.
We may also assume $\alpha\geq1$ since otherwise the right hand side
is $\gg_n T^n$ and \eqref{AGMLEMMAPF1} is trivial.
Now note that if $\vecx$ belongs to the set in the left hand side of
\eqref{AGMLEMMAPF1} then for every real vector $\vecy$ in the unit box
$\vecx+[0,1]^n$ we have
$\frac12T<\|\vecy\|_\infty\leq T+1\leq2T$ and (since all $x_j\geq1$)
\begin{align}
\prod_{j=1}^n y_j
\leq\prod_{j=1}^n(x_j+1)
\leq\prod_{j=1}^n(2x_j)
=2^n\prod_{j=1}^n x_j
<2^n(\|\vecx\|_\infty)^n\alpha^{-n}
\leq 2^nT^n\alpha^{-n}.
\end{align}
Hence the left hand side of \eqref{AGMLEMMAPF1} is
\begin{align}\notag
&\leq\vol\biggl(\biggl\{\vecy\in\R_{\geq1}^n\col
\sfrac12T<\|\vecy\|_\infty\leq 2T,\:
\prod_{j=1}^ny_j<2^nT^n\alpha^{-n}\biggr\}\biggr)
\hspace{100pt}
\\\notag
&\leq n\int_1^{2T}\cdots\int_1^{2T}\int_{\frac12T}^{2T}
I\biggl(\prod_{j=1}^ny_j<2^nT^n\alpha^{-n}\biggr)\,dy_n\,dy_{n-1}\,\cdots\,dy_1
\\
&\leq 2nT\int_1^{2T}\cdots\int_1^{2T}
I\biggl(\prod_{j=1}^{n-1}y_j<2^{n+1}T^{n-1}\alpha^{-n}\biggr)
\,dy_{n-1}\,\cdots\,dy_1
\\\notag
&=2^nnT^n\int_0^{\log(2T)}\cdots\int_0^{\log(2T)}
I\biggl(\sum_{j=1}^{n-1}u_j>\log(\alpha^n/4)\biggr) 
e^{-\sum_{j=1}^{n-1}u_j}
\,du_{n-1}\,\cdots\,du_1,
\end{align}
where in the last step we substituted $y_j=2Te^{-u_j}$.
If $n=2$ then the last expression is clearly $\ll T^2\alpha^{-2}$, as desired.
From now on we assume $n\geq3$.
Set $u_{n-1}=s+\log(\alpha^n/4)-\sum_{j=1}^{n-2}u_j$;
then the conditions $\sum_{j=1}^{n-1}u_j>\log(\alpha^n/4)$ and $u_{n-1}>0$
are equivalent with $s>0$ and $\sum_{j=1}^{n-2}u_j<s+\log(\alpha^n/4)$,
respectively.
Hence the last expression is
\begin{align}\notag
\leq 2^{n+2}nT^n\alpha^{-n}\int_0^\infty e^{-s}
\biggl(
\int_0^\infty\cdots\int_0^\infty 
I\biggl(\sum_{j=1}^{n-2}u_j<s+\log(\alpha^n/4)\biggr)\,du_{n-2}\,\cdots\, du_1
\biggr)\,ds
\\
\leq\frac{2^{n+2}n}{(n-2)!}
T^n\alpha^{-n}\int_0^\infty e^{-s}(s+n\log\alpha)^{n-2}\,ds
\ll_n T^n\alpha^{-n}(\log(2+\alpha))^{n-2},
\end{align}
where we used $\alpha\geq1$.
This completes the proof of the lemma.
\end{proof}

We now give the proof of \eqref{UNIFTHMCORRES1} in Corollary \ref{UNIFTHMCOR}.
We may assume $R\geq10$ since otherwise \eqref{UNIFTHMCORRES1}
follows immediately from $P_d(T,R)\leq1$.
We keep $R'\in[1,R]$, to be fixed later. Now
\begin{align}\notag
P_d(T,R) \ll_{d,\scrD} & T^{-d}\#\biggl\{
\veca\in\widehat\N^d\cap T\scrD\col\frac{f(\veca)}{s(\veca)}>R'
\:\text{ or }\:
\frac{s(\veca)}{(a_1\cdots a_d)^{1/(d-1)}}>\frac R{R'}\biggr\}
\\
\leq & T^{-d}\#\biggl\{
\veca\in\widehat\N^d\cap T\scrD\col\frac{f(\veca)}{s(\veca)}>R'\biggr\}
\\\notag
&\hspace{50pt}
+T^{-d}\#\biggl\{
\veca\in\NN^d\col\|\veca\|_\infty\leq\kappa_\scrD'T,\:
\frac{s(\veca)}{(a_1\cdots a_d)^{1/(d-1)}}>\frac R{R'}\biggr\},
\end{align}
where $\kappa_\scrD':=\sup_{\vecx\in\scrD}\|\vecx\|_\infty$.
In the last term, at the price of an extra factor $d$ we may 
impose the extra assumption $a_d=\max(a_1,\ldots,a_d)$. %
For such %
vectors $\veca$, we have
\begin{align}\notag
\frac{s(\veca)}{(a_1\cdots a_d)^{1/(d-1)}}
<\frac{d^{3/2}a_d\max(a_1,\ldots,a_n)}{\|\veca\|^{1-1/n}(a_1\cdots a_d)^{1/n}}
<\frac{d^{3/2}a_d\max(a_1,\ldots,a_n)}{a_d^{1-1/n}(a_1\cdots a_d)^{1/n}}
\hspace{30pt}
\\
=d^{3/2}\frac{\|(a_1,\ldots,a_n)\|_\infty}{(a_1\cdots a_n)^{1/n}}.
\end{align}
Hence for any $T>0$ with
$\widehat\N^d\cap T\scrD\neq\emptyset$,
\begin{align}\notag
P_d(T,R)\ll_{d,\scrD} T^{-d}\#\biggl\{
\veca\in\widehat\N^d\cap T\scrD\col\frac{f(\veca)}{s(\veca)}>R'\biggr\}
\hspace{150pt}
\\
+T^{-n}\#\biggl\{
\veca\in\N^n\col \|\veca\|_\infty\leq\kappa_\scrD'T,\:
\frac{\|(a_1,\ldots,a_n)\|_\infty}{(a_1\cdots a_n)^{1/n}}
>\frac{1}{d^{3/2}}\frac R{R'}\biggr\}
\hspace{20pt}
\\\notag
\ll_{d,\scrD} {R'}^{-n}+
R^{-n}{R'}^n\Bigl(\log\Bigl(2+\frac R{R'}\Bigr)\Bigr)^{n-2},
\end{align}
where we used Theorem \ref{UNIFTHM} and Lemma \ref{AGMLEMMA}.
The bound in \eqref{UNIFTHMCORRES1} now follows by choosing 
$R'=\sqrt R(\log(R+2))^{\frac1n-\frac12}$.
\hfill$\square$

\section{Lattice coverings of space with convex bodies}
\label{LATTICECOVSEC}

According to a theorem of Schmidt
(\cite[Thm.\ 11$^*$]{wS59}),
sharpening a previous result by Rogers
(\cite[Thm.\ 2]{cR58}),
if $n$ is sufficiently large, then for any $n$-dimensional
convex body $K$ of volume 
\begin{align}\label{SCHMIDTLOWERBOUND}
\vol_n(K)\geq(1+\eta_0)^n
\qquad (\text{with $\eta_0=0.756\ldots$ as in Theorem \ref{SECONDMAINTHM}}),
\end{align}
there exists a lattice $L\in X_n$ such that the translates of
$K$ by $L$ cover $\R^n$, viz.\ $K+L=\R^n$.
The lower bound \eqref{SCHMIDTLOWERBOUND} was shortly 
afterwards improved by Rogers
to a sub-exponential bound, 
in \cite{cR59}.
However, our purpose in this section is to point out that
the argument in \cite{wS59}, \cite{cR58} can fairly easily be modified to
give that $K+L=\R^n$ holds not just for \textit{some} lattice
$L\in X_n$, but in fact for a subset of large measure in $X_n$:
\begin{thm}\label{ROGERSPRECISETHM}
Let $\eta_0=0.756\ldots$ be the unique real root of $e\log \eta+\eta=0$.
For every dimension $n$ larger than a certain absolute constant,
if $a$ is any real number satisfying
\begin{align}\label{ROGERSPRECISETHMASS}
n\eta_0^{\frac12n}\leq a<1,
\end{align}
and $K$ is any $n$-dimensional convex body of volume
\begin{align}\label{ROGERSPRECISETHMASS2}
\vol_n(K)\geq n\bigl(1+\eta_0a^{-\frac1n}\bigr)^n,
\end{align}
then
\begin{align}\label{ROGERSPRECISETHMRES}
\mu_n\bigl(\bigl\{L\in X_n\col K+L=\R^n\bigr\}\bigr)\geq 1-a.
\end{align}
\end{thm}
In particular, for any given constant $\alpha>1+\eta_0$ there exists
$c<1$ such that for any sufficiently large $n$, and for any
convex body $K\subset\R^n$ of volume $\geq\alpha^n$,
the probability that $K$ fails to give a covering
with respect to a random lattice $L\in X_n$
is $\leq c^n$, i.e.\ exponentially small in $n$.
We obtain Theorem \ref{SECONDMAINTHM} as a special case of this by taking
$n=d-1$ and $K=\alpha(d-1)!^{\frac1{d-1}}\Delta$.

\subsection{Proof of Theorem \ref*{ROGERSPRECISETHM}}

We start by recalling another result of Rogers 
(\cite{cR58a})
which is used in the proof of 
\cite[Thm.\ 11$^*$]{wS59}.
For any (Lebesgue) measurable set $M\subset\R^n$ and any lattice
$L\in X_n$ we write $\epsilon(M,L)$ for the density of the set of points
in $\R^n$ left uncovered by the translates of $M$ by the vectors of $L$.
In other words,
\begin{align}
\epsilon(M,L) %
=1-\vol_n((M+L)/L).
\end{align}
(Note that $(M+L)/L$ is a well-defined measurable subset of the torus 
$\R^n/L$.)

\begin{thm} \label{ROGERSLIGHTTHM}
(\cite[Thm.\ 1]{cR58a}\footnote{The boundedness assumption in Rogers' statement of
\cite[Thm.\ 1]{cR58a}
can be disposed of, cf.\ 
\cite[p.\ 211]{cR58a}.
Note also that we do not have to require $V\leq1$, although
if $V>1$ then the bound in \eqref{ROGERSLIGHTTHMRES} is subsumed by the bound
$\int\epsilon(M,L)\,d\mu_n\leq\frac12$
which follows by applying Theorem \ref{ROGERSLIGHTTHM} to an arbitrary subset
$M'\subset M$ of volume $1$.}
For any %
measurable set $M\subset\R^n$ ($n\geq2$) of volume $V$,
\begin{align}\label{ROGERSLIGHTTHMRES}
\int_{X_n}\epsilon(M,L) %
\,d\mu_n(L)
\leq1-V+\sfrac12V^2.
\end{align}
\end{thm}

Let us note the following corollary.
\begin{cor}\label{ROGERSLIGHTCOR}
For any $C>0$ and any %
measurable set $M\subset\R^n$ ($n\geq2$) of volume $V$,
\begin{align}\label{ROGERSLIGHTCORRES}
\mu_n\bigl(\bigl\{L\in X_n\col \epsilon(M,L)\geq 1-V+CV^2\bigr\}\bigr)\leq
\frac1{2C}.
\end{align}
\end{cor}
\begin{proof}
Clearly, for \textit{any} lattice $L\in X_n$ we have
$\vol_n((M+L)/L)\leq V$, and thus
\begin{align}
\epsilon(M,L)\geq1-V.
\end{align}
Hence if $p$ denotes the measure in the left hand side of 
\eqref{ROGERSLIGHTCORRES} then
\begin{align}
\int_{X_n}\epsilon(M,L)\,d\mu_n(L)\geq p(1-V+CV^2)+(1-p)(1-V)
=1-V+pCV^2,
\end{align}
and thus Theorem \ref{ROGERSLIGHTTHM} implies $pC\leq\frac12$.
\end{proof}

\begin{proof}[Proof of Theorem \ref{ROGERSPRECISETHM}]
Let $a$ and $K$ be given as in the statement of the theorem.
Let $r=0.278\ldots$ be the root of the equation $1+r+\log r=0$;
then $\eta_0=e^{-r}$.
We set $K'=\rho K$, where $\rho>0$ is chosen so that 
the volume of $K'$ is
\begin{align}
V=\vol_n(K')=rn.
\end{align}
We also set 
\begin{align}
\eta=e^{-r}a^{-\frac1n}=\eta_0a^{-\frac1n}.
\end{align}

Now by Schmidt
\cite[Thm.\ 10$^*$]{wS59} (applied with $\ve=1$),
if $n$ is larger than a certain absolute constant then
\begin{align}\label{XNINTBOUND}
\int_{X_n}\epsilon(K',L)\,dL\leq 2
(1+V^{n-1}n^{-n+1}e^{V+n})e^{-V}
=2(1+r^{-1})e^{-rn},
\end{align}
and thus
\begin{align}\label{XNINTBOUND2}
\mu_n\bigl(\bigl\{L\in X_n\col \epsilon(K',L)\geq4(1+r^{-1})e^{-rn}a^{-1}
\bigr\}\bigr)\leq\sfrac12 a.
\end{align}
Also by Corollary \ref{ROGERSLIGHTCOR},
\begin{align}\label{XNINTBOUND3}
\mu_n\bigl(\bigl\{L\in X_n\col \epsilon(\eta K',L)\geq 1-\eta^nV+
a^{-1}\eta^{2n}V^2\bigr\}\bigr)\leq\sfrac12a.
\end{align}
Note that $\frac{e^{-rn}a^{-1}}{\eta^nV}=\frac 1V=r^{-1}n^{-1}\to0$
as $n\to\infty$, and also
\begin{align}\label{CUTOFFQUOTIENT1}
\frac{a^{-1}\eta^{2n}V^2}{\eta^n V}
=a^{-1}\eta^nV
=a^{-2}e^{-rn}rn\leq rn^{-1}\to0\qquad\text{as }\: n\to\infty,
\end{align}
where we used \eqref{ROGERSPRECISETHMASS}. %
Hence for $n$ larger than a certain absolute constant, we have
\begin{align}\label{CUTOFFSUM}
1-\eta^nV+a^{-1}\eta^{2n}V^2+
4(1+r^{-1})e^{-rn}a^{-1}<1.
\end{align}

It follows from \eqref{XNINTBOUND2}, \eqref{XNINTBOUND3} and
\eqref{CUTOFFSUM} that
\begin{align}\label{THM3P1PF1}
\mu_n\bigl(\bigl\{L\in X_n\col \epsilon(\eta K',L)+\epsilon(K',L)<1
\bigr\}\bigr)\geq1-a.
\end{align}
However, for any $L\in X_n$ satisfying
$\epsilon(\eta K',L)+\epsilon(K',L)<1$
we have $(1+\eta)K'+L=\R^n$,
since $K'$ is convex
(cf.\ 
\cite[Sec.\ 1.3]{cR58}),
and thus also $\alpha K'+L=\R^n$ for any $\alpha\geq1+\eta$.
In particular, since $K=\rho^{-1}K'$, 
$K+L=\R^n$ holds for any such $L$, provided that
we have $\rho^{-1}\geq1+\eta$.
But $\vol_n(K)=\rho^{-n}V$; hence
$\rho^{-1}\geq1+\eta$ is equivalent with
$\vol_n(K)\geq(1+\eta)^nV$,
and this inequality certainly holds, because of $V<n$ and
our assumption \eqref{ROGERSPRECISETHMASS2}.
Hence \eqref{ROGERSPRECISETHMRES} follows from \eqref{THM3P1PF1}.
\end{proof}

\end{document}